\documentclass[10pt,a4paper]{article}
\usepackage{amsmath, amssymb, amsbsy, amsthm}
\usepackage{amsfonts}
\usepackage{graphicx}
\usepackage{epstopdf}
\usepackage{verbatim}
\usepackage{hyperref}
\usepackage{ifthen}
\usepackage{enumerate}
\usepackage{esint}
\usepackage{authblk}
\usepackage{subfigure}

\makeatletter
\newcommand{\logmessage}[1]{\@latex@warning{#1}}
\makeatother
\numberwithin{equation}{section}
\theoremstyle{plain}
\newtheorem{thm}{Theorem}[section]
\newtheorem{prop}[thm]{Proposition}
\newtheorem{rem}[thm]{Remark}
\newtheorem{ass}[thm]{Assumption}
\newtheorem{lem}[thm]{Lemma}
\newtheorem{cor}[thm]{Corollary}

\theoremstyle{definition}
\newtheorem{definition}[thm]{Definition}
\newtheorem{example}[thm]{Example}
\newtheorem{remark}[thm]{Remark}

\newcommand{\rev}[1]{#1}
\newcommand{\Fcal}{\mathcal{F}}
\newcommand{\lr}[1]{\left(  #1 \right)}

\newcommand{\Gcal}{\mathcal{G}}

\newcommand{\Fe}{\Fcal_{\epsilon}}

\newcommand{\Ge}{\mathcal{G}_{\epsilon}}
\newcommand{\Fef}{\Fcal_{\epsilon,f}}
\newcommand{\Ged}{\Ge^\diamond}
\newcommand{\Gd}{\mathcal{G}^\diamond}
\newcommand{\open}[2]{Open_{#1}\lr{#2}}
\newcommand{\close}[2]{Close_{#1}\lr{#2}}

\renewcommand{\H}{\mathcal{H}}

\newcommand{\field}[1]{\ensuremath{\mathbb{#1}}}
\newcommand{\R}{\field{R}}
\newcommand{\N}{\field{N}}

\newcommand{\abs}[1]{\left \lvert#1 \right\rvert}

\newcommand{\set}[1]{\left\{#1\right\}}

\newcommand{\ve}{\varepsilon}

\newcommand{\hn}{H_0^1(\Omega)}

\newcommand{\hvdc}{H_{\diamond,c}^1(\Omega)}
\newcommand{\bvdo}{\text{BV}_{\diamond,1}}

\newcommand{\gs}{\geqslant}
\newcommand{\ls}{\leqslant}

\newcommand{\per}{\operatorname{Per}}
\newcommand{\supp}{\operatorname{Supp}}
\renewcommand{\div}{\operatorname{div}}
\newcommand{\dd}{\, \mathrm{d}}
\newcommand{\bigring}[1]{\overset{\circ}{#1}}

\makeatletter
\newcommand{\ignore}{\logmessage{Text ignored}\@gobble}
\makeatother

\makeatletter
\let\@fnsymbol\@alph
\makeatother

\title{Critical yield numbers of rigid particles settling in Bingham fluids and Cheeger sets}
\author[1]{Ian A. Frigaard \thanks{frigaard@math.ubc.ca}}
\author[2]{Jos\'{e} A. Iglesias \thanks{jose.iglesias@ricam.oeaw.ac.at}}
\author[2]{Gwenael Mercier \thanks{gwenael.mercier@ricam.oeaw.ac.at}}
\author[3]{Christiane P\"{o}schl \thanks{christiane.poeschl@aau.at}}
\author[2,4]{Otmar Scherzer \thanks{otmar.scherzer@univie.ac.at}}
\affil[1]{\small Department of Mathematics and Department of Mechanical Engineering, University of British Columbia, Vancouver, BC, Canada.}
\affil[2]{\small Johann Radon Institute for Computational and Applied Mathematics (RICAM), Austrian Academy of Sciences, Linz, Austria.}
\affil[3]{\small Universit\"at Klagenfurt, Klagenfurt, Austria.}
\affil[4]{\small Computational Science Center, University of Vienna, Vienna, Austria.}
\date{}
\begin{document}

\maketitle


\begin{abstract}
We consider the fluid mechanical problem of identifying the critical yield number $Y_c$ of a dense solid inclusion (particle) settling under gravity within a bounded domain of Bingham fluid, i.e.~the critical ratio of yield stress to buoyancy stress that is sufficient to prevent motion.
We restrict ourselves to a two-dimensional planar configuration with a single anti-plane component of velocity. Thus, both particle and fluid domains are infinite cylinders of fixed cross-section. We then show that such yield numbers arise from an eigenvalue problem for a constrained total variation.
We construct particular solutions to this problem by consecutively solving two Cheeger-type set optimization problems. Finally, we present a number of example geometries in which these geometric solutions can be found explicitly and discuss general features of the solutions. 
\end{abstract}

\section{Introduction}
\label{sec:intro}

100 years ago Eugene Bingham \cite{Bingham1916} presented results of flow experiments through a capillary tube, measuring the flow rate and pressure drop for various materials of interest. Unlike with simple viscous fluids, he recorded a ``friction constant'' (a stress) that must be exceeded by the pressure drop in order for flow to occur, and thereafter postulated a linear relationship between applied pressure drop and flow rate. This empirical flow law evolved into the Bingham fluid: the archetypical yield stress fluid. However, it was not until the 1920's that ideas of visco-plasticity became more established \cite{Bingham1922} and other flow laws were proposed e.g.~\cite{Herschel-Bulkley1926}. These early works were empirical and focused largely at viscometric flows. Proper tensorial descriptions, general constitutive laws and variational principles waited until Oldroyd \cite{Oldroyd1947} and Prager \cite{Prager1954}. These constitutive models are now widely used in a range of applications, in both industry and nature; see \cite{Balmforth2014} for an up to date review.

An essential feature of Bingham fluids flows is the occurrence of plugs: that is regions within the flow containing fluid that moves as a rigid body. This occurs when the deviatoric stress falls locally below the yield stress, which is a physical property of the fluid. Plug regions may occur either within the interior of a flow or may be attached to the wall. In general, as the applied forcing decreases, the plug regions increase in size and the velocity decreases in magnitude. It is natural that at some critical ratio of the driving stresses to the resistive yield stress of the fluid, the flow stops altogether. This critical yield ratio or \emph{yield number} is the topic of this paper.

Critical yield numbers are found for even the simplest 1D flows, such as Poiseuille flows in pipes and plane channels or uniform film flows, e.g.~paint on a vertical wall. These limits have been estimated and calculated exactly for flows around isolated particles, such the sphere \cite{Beris1985} (axisymmetric flow) and the circular disc \cite{Randolph1984,Tokpavi2008} (2D flow). Such flows have practical application in industrial non-Newtonian suspensions, e.g.~mined tailings transport, cuttings removal in drilling of wells, etc.

The first systematic study of critical yield numbers was carried out by Mosolov \& Miasnikov \cite{MosMia65,MosMia66} who considered anti-plane shear flows, i.e.~flows with velocity $\mathbf{u} = (0,0,w(x_1,x_2))$ in the $x_3$-direction along ducts (infinite cylinders) of arbitrary cross-section $\Omega$. These flows driven by a constant pressure gradient only admit the static solution ($w(x_1,x_2) = 0$) if the yield stress is sufficiently large. Amongst the many interesting results in \cite{MosMia65,MosMia66} the key contributions relate to exposing the strongly geometric nature of calculating the critical yield number $Y_c$. Firstly, they show that $Y_c$ can be related to the maximal ratio of area to perimeter of subsets of $\Omega$. Secondly, they develop an algorithmic methodology for calculating $Y_c$ for specific symmetric $\Omega$, e.g.~rectangular ducts. This methodology is extended further by \cite{Huilgol2006}.

Critical yield numbers have been studied for many other flows, using analytical estimates, computational approximations and experimentation. Critical yield numbers to prevent bubble motion are considered in \cite{Dubash2004,Tsamopoulos2008}. Settling of shaped particles is considered in \cite{Jossic2001,Putz2010}. Natural convection is studied in \cite{Karimfazli2013,Karimfazli2015}. The onset of landslides are studied in  \cite{Hild2002,Ionescu2005,Hassani2005} (where the terminologies ``load limit analysis'' and ``blocking solutions'' have also been used). In \cite{Frigaard1998,Frigaard2000} we have studied two-fluid anti-plane shear flows, that arise in oilfield cementing.

In this paper we study critical yield numbers for two-phase anti-plane shear flows, in which a particulate solid region ${\Omega}_s$ settles under gravity in a surrounding Bingham fluid of smaller density. As the particle settles downwards the surrounding fluid moves upwards, with zero net flow: a so called \emph{exchange flow}. Our objective is to derive new results that set out an analytical framework and algorithmic methodology for calculating $Y_c$ for this class of flows. 

Our analysis naturally leads to the so-called Cheeger sets, that is, minimizers of the ratio of perimeter to volume inside a given domain. Recently, starting with \cite{KawFri03}, many of their properties have been studied, particularly regularity and uniqueness in the case of convex domains \cite{KawLac06, CasChaNov10}. These sets constitute examples of explicit solutions to the total variation flow, which has motivated their investigation \cite{AndCasMaz04, BelCasNov02, BelCasNov05}.

A related line of research is the use of total variation regularization in image processing. In particular, set problems like those treated here appear in image segmentation \cite{ChaEseNik06} and as the problem solved by the level sets of minimizers \cite{ChaCasCreNov10,All09,CasNovPoe16} of the Rudin Osher Fatemi functional \cite{RudOshFat92}. The analogy between anti-plane shear flows of yield stress fluids and imaging processing techniques has been exploited previously by the authors in the context of nonlinear diffusion filtering using total variation flows or bounded variation type regularization. In our previous work \cite{Frigaard2003,Frigaard2006} we exploited physical insights from the fluid flow problem in order to derive optimal stopping times for diffusion filtering. 

\subsection{Summary and outline}
First, in Section \ref{sec:model} we write the simplified Navier-Stokes equations and corresponding variational formulation for the inclusion of a Newtonian fluid in a Bingham fluid, in geometries consisting of infinite cylinders and anti-plane velocities.\\
Section \ref{se:energy} is dedicated to the background theory for the exchange flow problem. After proving existence of solutions, we make the viscosity of the inclusion tend to infinity, that is, we study the flow of a solid inclusion into a Bingham fluid.\\
We then recall the usual notion of critical yield number, seen as the supremum of an eigenvalue quotient \eqref{eq:eigenvalue} in the standard Sobolev space $H^1$, which writes after simplification as a minimization of total variation with constraints. Since it is well known that such a problem does not necessarily have a solution in $H^1$, we relax it enlarging the admissible space to functions with bounded variation, which ensures the existence of a minimizer. \\
In Section \ref{se:pwcm} we study the relaxed problem and show that we can construct minimizers that attain only three values, and whose level-sets are solutions of simple geometrical problems closely related to the Cheeger problem (see Def. \ref{de:ch}). We show how the geometrical properties of Cheeger sets are reflected in the structure of our three level-set minimizer, and give several explicit examples exhibiting the influence of the geometry of the domain and the particles in that of the solution. In particular, we emphasize the role of non-uniqueness of Cheeger sets in the non uniqueness of our minimizers. \\
Finally, Section \ref{se:examples} is dedicated to the explicit construction of three-valued solutions and computing the corresponding yield numbers in simple situations.\\
It has to be noticed that the restriction to anti-plane flows and equal particle velocities is fundamental in all this work. The in-plane case remains an exciting challenge.


\section{Modelling}
\label{sec:model}
As discussed in Section \ref{sec:intro} we study anti-plane shear flows of particles within a Bingham fluid. Anti-plane shear flows have velocity in a single direction and the velocity depends on the 2 other coordinate directions. We assume the solid is denser than the fluid ($\hat{\rho}_f < \hat{\rho}_s$) and align the flow direction $\hat{x}_3$ with gravity.
In the anti-plane shear flow context,  particles (solid regions) are infinite cylinders represented as ${\hat{\Omega}}_s \times \R \subseteq \R^3$ and moving uniformly in the $\hat{x}_3$-direction. The flows are thus described in a two-dimensional region $(\hat{x}_1,\hat{x}_2) \in \hat{\Omega}$. The fluid is contained in
$({\hat{\Omega}}_f := {\hat{\Omega}} \backslash {\hat{\Omega}}_s) \times \R$, and is considered to be a Bingham fluid.
The flow variables are the deviatoric stress $\hat{\mathbf{\tau}}$, pressure $\hat{p}$ and velocity $\hat{w}$, all of which are independent of $\hat{x}_3$. Only steady flows are considered.

The fluid is characterized physically by its density, yield stress and plastic viscosity: $\hat{\rho}_f$, $\hat{\mu}_f$ and $\hat{\tau}_Y$, respectively. We adopt a fictitious domain approach to modelling the solid phase, treating it initially as a fluid and then formally taking the solid viscosity to infinity. The solid phase density and viscosity are $\hat{\rho}_s$ and $\hat{\mu}_s$. These parameters are assumed constant.

The incompressible Navier-Stokes equations simplify to only the $\hat{x}_3$-momentum balance. This and the constitutive laws are:
\begin{equation}
\hat{\div}\, \hat{\mathbf{\tau}} =
\begin{cases}
\hat{p}_{x_3} - \hat{\rho}_f \hat{g} \quad \text{ in } \hat
{\Omega}_f\,,\\
\hat{p}_{x_3} - \hat{\rho}_s \hat{g} \quad \text{ in } \hat
{\Omega}_s\,,\\
\end{cases}
\quad
\hat{\mathbf{\tau}}  =
\begin{cases}
\left( \hat{\mu}_f + \frac{\hat{\tau}_Y}{\abs{\hat{\nabla}
\hat{w}}}
\right) \hat{\nabla} \hat{w}
&\text{in }  \hat{\Omega}_f\,,\\
\hat{\mu}_s \hat{\nabla} \hat{w}
&\text{in } \hat{\Omega}_s\;,
\end{cases}
\label{x3_momentum}
\end{equation}
where $\hat{g}$ is the gravitational acceleration. Strictly speaking the fluid constitutive law applies only to where $|\hat{\mathbf{\tau}}| > \hat{\tau}_Y$.

The above model and variables are dimensional, for which we have adopted the convention of using the ``hat'' accent, e.g.~$\hat{g}$. We now make the model dimensionless by scaling. 
In (\ref{x3_momentum}) the driving force for the motion is the density difference, which results in a buoyancy force that scales proportional to the size of the particle. Thus, we scale lengths with $\hat{L}$:
\begin{equation*}
 \hat{L} = \sqrt{\text{area}({\hat{\Omega}}_s)}\,, \quad \mathbf{x} =(x_1,x_2) :=
\frac{1}{\hat{L}}(\hat{x}_1,\hat{x}_2)\,, \quad \nabla = \hat{L} \hat{\nabla}\,, \quad \div = \hat{L}\,\hat{\div}.
\end{equation*}
An appropriate measure of the buoyancy stress is $(\hat{\rho}_s- \hat{\rho}_f)\hat{g}\hat{L}$, which we use to scale $\hat{\mathbf{\tau}} = (\hat{\rho}_s- \hat{\rho}_f)\hat{g}\hat{L} \mathbf{\tau}$. For the pressure gradient in (\ref{x3_momentum}) we subtract the hydrostatic pressure gradient from the fluid phase and scale the modified pressure gradient with $(\hat{\rho}_s -
\hat{\rho}_f)\hat{g}$, defining:
\[
f  =  \frac{\hat{p}_{x_3}-\hat{\rho}_f \hat{g}}{ (\hat{\rho}_s - \hat{\rho}_f)\hat{g}}  .
\]
The scaled momentum equations are:
\begin{equation}
\div \mathbf{\tau} =
\begin{cases}
f \quad \text{ in } \Omega_f\,,\\
f-1 \quad \text{ in } \Omega_s\,,\\
\end{cases}
\end{equation}

For the constitutive laws, we define a velocity scale $\hat{w}_0$ by balancing the buoyancy stress with a representative viscous stress in the fluid:
\[ (\hat{\rho}_s- \hat{\rho}_f)\hat{g}\hat{L} = \frac{\hat{\mu}_f\hat{w}_0}{\hat{L}} . \]
Scaled constitutive laws are:
\begin{equation}
\mathbf{\tau} = \frac{1}{\varepsilon} \nabla  w \text{ in } \Omega_s;
\quad
\begin{cases}
\mathbf{\tau} = \left( 1 + \displaystyle{\frac{Y}{\abs{\nabla w}}}
\right) \nabla w
&\,   \abs{\mathbf{\tau}} > Y,  \\
\abs{\nabla w} = 0
&\, \abs{\mathbf{\tau}} \le Y
\end{cases}\quad\text{in } \Omega_f.
\label{x3_momentum_dimensionless}
\end{equation}
We note that there are two dimensionless parameters: $\varepsilon$ and $Y$, defined as:
\begin{equation*}
\varepsilon:= \frac{\hat{\mu}_f}{\hat{\mu}_s} \,, \quad  \quad
Y:= \frac{\hat{\tau}_Y }{ (\hat{\rho}_s-\hat{\rho}_f)\hat{g}\hat{L}} .
\end{equation*}
Evidently, $\varepsilon$ is a viscosity ratio. Soon we shall consider the solid limit $\varepsilon \to 0$, and thereafter $\varepsilon$ plays no role in our study.

The parameter $Y$ is called the \emph{yield number} and is central to our study. We see that physically $Y$ balances the yield stress and the buoyancy stress. As buoyancy is the only driving force for motion, it is intuitive that there will be no flow if $Y$ is large enough. The smallest $Y$ for which the motion is stopped is called the \emph{critical yield number}, $Y_c$, although this will be defined rigorously later.\footnote{The yield number is sometimes referred to as the yield gravity number or yield buoyancy number. As the viscous stresses are also driven by buoyancy, an alternate interpretation would be as a ratio of yield stress to viscous stress, which is referred to as the Bingham number.}

In terms of $w$ the momentum equation is:
\begin{equation}
\label{eq:Pde}
\begin{array}{rlc}
\div \left( \left( 1 + \frac{Y}{\abs{\nabla w}}\right) \nabla
w\right) &= f &\text{in } \Omega_f, \\
\div \left( \frac{1}{\varepsilon} \nabla  w\right) &= f-1
&\text{
in } \Omega_s\;.
\end{array}
\end{equation}
It is assumed that $\Omega$ has finite extent and at the stationary boundary we assume the no-slip condition:
\begin{equation} \label{eq:bcd}
   w \equiv 0 \text{ on } \partial \Omega\, .
\end{equation}
At the interface between the two phases the shear stresses are assumed continuous, leading to the transmission condition:
\begin{equation}\label{eq:transC}
  \frac{1}{\epsilon} \nabla w \cdot \mathbf{n}_s +
  \lr{1+\frac{Y}{\abs{\nabla w}}}
  \nabla w \cdot \mathbf{n}_f = 0
  \quad \text{on }\partial \Omega_s.
\end{equation}
Here $\mathbf{n}_s,~\mathbf{n}_f$ denote the outer unit-normals on $\partial \Omega_s,~\partial \Omega_f$, and the equality has to hold in a weak sense.

We note that for given $f$ and $\ve > 0$ fixed, the solution \rev{$w_f$}
of \eqref{eq:Pde}, \eqref{eq:transC}, \eqref{eq:bcd} is equivalently
characterized
as the minimizer of the functional
\begin{equation}
\label{eq:Fef}
\begin{aligned}
  \Fef(w) &:= \Ge(w) +  \int_\Omega \rev{f} w \text{ with }\\
  \Ge(w)&:=\frac{1}{2}\int_{\Omega_f} \abs{\nabla w}^2 +
    \frac{1}{2 \varepsilon} \int_{\Omega_s} \abs{\nabla w}^2 +
    Y \int_{\Omega_f} \abs{\nabla w}  - \int_{\Omega_s} w
\end{aligned}
\end{equation}
over the space $\hn$.

\section{Exchange Flow Problem}
\label{se:energy}

Physically, as a solid particle settles in a large expanse of incompressible fluid, its downwards motion causes an equal upwards motion such that the net volumetric flux is zero. Here we wish to mimic this same scenario in the anti-plane shear flow context. 

Therefore, we are interested in the \emph{exchange flow problem}, which \rev{consists in finding the pair $(w,f)$ that satisfies:}
\begin{itemize}
\item Equation \eqref{eq:Pde} \rev{and condition \eqref{eq:transC} in a suitable variational sense},
\item the homogeneous boundary conditions \eqref{eq:bcd},
\item and the \emph{exchange flow condition}
      \begin{equation} \label{eq:ic}
            \int_\Omega w(x)\,dx =0\;.
      \end{equation}
\end{itemize}
\noindent
\rev{Note that \eqref{eq:ic} states that the anti-plane flow is divergence free. Therefore, we identify $f$ with a scalar.} Two equivalent formulations of this problem are possible:
\begin{enumerate}
\item Finding a saddle point of the functional
      \begin{equation}\label{eq:Fvf}
      \Fe(w,f):=\Fef(w)
      \end{equation}
      on $\hn \times \R$, with $\Fef$ from \eqref{eq:Fef}.
      In other words, $f$ is a Lagrange multiplier in the saddle point problem
      for satisfying the constraint \eqref{eq:ic}.
\item Incorporating the constraint \eqref{eq:ic} as part of the domain of definition.
      Thus we consider minimization of the functional
      \begin{equation}\label{eq:Fepsilon2}
	\Ged(w):=
	\begin{cases}\Ge(w) &
	 \text{if }w\in H_\diamond^1(\Omega):=\set{w\in \hn: \int_\Omega w= 0}\,,\\
	+\infty &\text{ for } w\in \hn \backslash H_\diamond^1(\Omega)\;.
	\end{cases}
      \end{equation}
      We show in Lemma \ref{le:ExistenceFfe} that a minimizer of $\Ged$ exists.
\end{enumerate}

In the rest of the paper we focus on the second formulation.

\begin{lem}\label{le:ExistenceFfe}
The functionals $\Fef(\cdot)$ and $\Ged(\cdot)$ attain their minimum.
If the minimizer $w^*$ of $\Fef(\cdot)$ satisfies $\int_\Omega w^* =
0$,
then it is also a minimizer of $\Ged(\cdot)$.
\end{lem}
\begin{proof}
In order to prove the existence of a minimizer of
$w\rev{\mapsto} \Fcal_{\epsilon,f}(w)$ for $f$ fixed, we show that the
functional is coercive and lower semi-continuous:

\begin{enumerate}[i)]
\item \textit{The functional $\Fcal_{\epsilon,f}(w)$ is coercive with
respect to
$w$.} For all $\delta > 0$, and denoting by $|\Omega|$ the Lebesgue measure of $\Omega$, it follows from \rev{Poincare and Jensen's inequalities} that
  \begin{equation}
   \label{eq:f_help}
  \begin{aligned}
  \rev{f} \int_\Omega w &\gs - \frac{1}{2\delta^2} \rev{f}^2 -
\frac{\delta^2}{2} \lr{\int_\Omega \abs{w}}^2
  \gs - \frac{1}{2\delta^2} \rev{f}^2 -
\frac{\delta^2}{2} |\Omega| \int_\Omega \abs{w}^2\\
 &\gs - \frac{1}{2\delta^2} \rev{f}^2 - C
\frac{\delta^2}{2} |\Omega| \int_\Omega \abs{\nabla w}^2,
   \end{aligned}
  \end{equation}
  \rev{similarly, we have}
  \begin{equation}
  \label{eq:eins_help}
  \begin{aligned}
  - \int_{\Omega_s} w &\gs \rev{- \frac{1}{2\delta^2}  - \frac{\delta^2}{2} |\Omega_s| \int_{\Omega_s} \abs{w}^2 \gs - \frac{1}{2\delta^2}  - \frac{\delta^2}{2} |\Omega| \int_{\Omega} \abs{w}^2 } \\
  &\gs - \frac{1}{2\delta^2}  - C \frac{\delta^2}{2} |\Omega| \int_\Omega \abs{\nabla w}^2\;.
  \end{aligned}
  \end{equation}
  \rev{Summing \eqref{eq:f_help} and \eqref{eq:eins_help} yields}
  \begin{equation*}
    f \int_\Omega w - \int_{\Omega_s} w \gs - \frac{1}{2\delta^2}
(f^2+1) - C \delta^2 |\Omega| \int_\Omega \abs{\nabla w}^2\;.
  \end{equation*}
  Now, choosing $\delta>0$ such that
  \begin{equation*}
    0 < C \delta^2 |\Omega| < \frac{1}{2} \min
\set{1,\frac{1}{\epsilon}}\,,
  \end{equation*}
  the coercivity with respect to $w$ follows.

\item  For $\epsilon<1$, we now have $2 C |\Omega| < 1/\delta^2$ and thus we see that $\Fcal_{\epsilon,f}$ is
bounded from below by $-C \, (f^2+1) |\Omega|$.

\item  \textit{The functional $\Fcal_{\epsilon,f}$ is weakly lower
semi-continuous:}
    $\Fcal_{\epsilon,f}$ can be rewritten as
    \begin{equation*}
     \Fcal_{\epsilon,f}(w) = \int_\Omega g(x,w(x),\nabla w(x))
dx\,,
    \end{equation*}
    where $p\rightarrow g(s,z,p)$ is convex. Since
    $\Fcal_{\epsilon,f}$ is also bounded below, we have (see for instance \cite[Thm. 13.1.2]{AttButMic06}) that $\Fcal_{\epsilon,f}(w)$ is weakly lower semi-continuous.
\end{enumerate}
With this (coercivity, boundedness and weak lower semi-continuity) existence of
a minimizer of $w\rightarrow \Fcal_{\epsilon,f}(w)$ follows
immediately (see \cite[Thm. 3.2.1]{AttButMic06}).

The proof of existence of minimizer of $\Fcal_\epsilon^\diamond$ requires in
addition to show that $H^1_{\diamond}(\Omega)$ is weakly closed.
Therefore note first that the set $H^1_{\diamond}(\Omega)$ is
convex (linearity of the \rev{constraint}) and closed with respect
to the norm topology on $H^1_{\diamond}(\Omega)$.
From this we can conclude that $H^1_{\diamond}(\Omega)$ is weakly closed, so
that (see \cite[Thm. 3.3.2]{AttButMic06}) the functional attains a
minimium on this subset.
\end{proof}
\subsection{Solid limit}
Now we want to study the behavior of the problem when $\hat{\mu}_s\rightarrow \infty$ (so that $\hat \Omega_s$ becomes rigid), that is, $\epsilon \rightarrow 0$.
We will see that it leads to minimization of the functional
      \begin{equation}\label{eq:F}
      \begin{aligned}
	\Gd: \hn &\rightarrow \R\cup \set{+\infty}\;.\\
	w &\rightarrow
	\begin{cases}
	\frac{1}{2}\int_{\Omega_f} \abs{\nabla w}^2 +
	 Y \int_{\Omega_f} \abs{\nabla w}  - \int_{\Omega_s} w
&\text{if} \; w \in  \hvdc \\
	+\infty &\text{else}
	\end{cases}
	\end{aligned}
      \end{equation}
      where we define
      \begin{equation*}
      \hvdc:=\set{w \in \hn: \int_\Omega w = 0, \; \nabla w = 0 \;
\text{in } \Omega_s}\;.
      \end{equation*}
\begin{lem}\label{le:gammaConvergence}
The functionals $\Ged$ defined in \eqref{eq:Fepsilon2} $\Gamma-$converge to
$\Gd$ in $\hn$, that is\rev{,}
for all $w\in \hn$ and all sequences $\set{\epsilon_j}_{j\in \N}$
converging to $0$ we have:
\begin{enumerate}[i)]
 \item (lim inf inequality) For every sequence $\set{w_j}_{j\in \N}$
converging
to $w$ in $H^1$ we have
        \begin{equation*}
         \Gd(w) \ls \lim \inf_{j \to \infty} \Gd_{\epsilon_j}(w_j).
       \end{equation*}
\item (lim sup inequality) There exists a sequence $\set{w_j}_{j\in \N}$
converging to $w$ in $H^1$ with
        \begin{equation}
        \label{eq:Gamma_sup}
         \Gd(w) \gs \lim \sup_{j \to \infty} \Gd_{\epsilon_j}(w_j)\;.
        \end{equation}
\end{enumerate}
\end{lem}

\begin{proof}
Let $w \in \hn$ and let $\epsilon_j \to 0+$ be a decreasing
sequence with limit $0$.
\begin{enumerate}[i)]
\item
 For every sequence $w_j$ converging to $w$ in $\hn$, we
have
\begin{equation*}
 \begin{gathered}
  \lim_{j\rightarrow \infty} \int_{\Omega} \abs{w_j} = \int_{\Omega} \abs{w},\quad
  \lim_{j\rightarrow \infty} \int_\Omega \abs{\nabla w_j}^2 = \int_\Omega \abs{\nabla w}^2,\\
  \lim_{j\rightarrow \infty} \int_\Omega w_j = \int_\Omega w, \quad
  \lim_{j\rightarrow \infty} \int_{\Omega_s} w_j = \int_{\Omega_s} w,
 \end{gathered}
\end{equation*}
such that for all $w \in \hvdc$
\begin{equation*}
\begin{aligned}
  \Gd(w)&= \frac{1}{2}\int_{\Omega_f} \abs{\nabla w}^2 +
	 Y \int_{\Omega_f} \abs{\nabla w}  - \int_{\Omega_s} w
	\\
 &\ls  \liminf_{j\rightarrow \infty}
 \lr{\frac{1}{\epsilon_j} \int_{\Omega_s}\abs{\nabla w_j}^2
+
 \frac{1}{2}\int_{\Omega_f} \abs{\nabla w_j}^2 +
	 Y \int_{\Omega_f} \abs{\nabla w_j}  - \int_{\Omega_s} w_j }\\
	 &\ls   \liminf_{j\rightarrow \infty}
\Gcal_{\epsilon_j}^\diamond(w_j)\;.
 \end{aligned}
\end{equation*}

If $w$ is not constant in $\Omega_s$,
\rev{$\Gcal^{\diamond}(w) =+\infty$} and also
$\liminf_{j\rightarrow \infty} \rev{\Gd_{\epsilon_j}(w_j)} \to \infty$
since $\lim_j \int_{\Omega_s} \abs{\nabla w_j}^2\not=0$ \rev{so} that
$\frac{1}{\epsilon_j}\int_{\Omega_s} \abs{\nabla w_j}^2 \rightarrow
\infty$.

\item
In the case where $w \not \in \hvdc$, we have
$$\limsup \Gcal_{\epsilon_j}^\diamond(w) = \infty
= \Gd(w).$$ For $w \in \hvdc$ we have that
$\int_{\Omega_s} \abs{\nabla w}^2=0.$
This shows that the constant sequence $w_j \equiv w$ satisfies \eqref{eq:Gamma_sup}.
\end{enumerate}
\end{proof}

Since \rev{$\Ged(\cdot)\gs \Gd_2(\cdot)$, they are equicoercive and} we get (see \cite[Thm. 1.21]{Bra02}) that

\begin{cor}The sequence of minimizers of $\Ged(\cdot)$ converges strongly in \rev{$H^1_0$} to the minimizer of $\Gd(\cdot)$ as $\epsilon \to 0$.
\end{cor}

\subsection{Critical yield numbers and total variation minimization}

We now want to identify the limiting yield number $Y$ such that the solution of the exchange flow problem satisfies
$w\equiv 0$ in $\Omega$, i.e.~both solid and fluid motions are stagnating.

\begin{definition}
The critical yield number is defined to be
\begin{equation}
\label{eq:eigenvalue}
Y_c := \sup_{\hvdc} \frac{\int_{\Omega_s} v}{\int_\Omega \abs{\nabla v}}\;.
\end{equation}
\end{definition}
Assume that $w_c$ minimizes $\Gd$, defined in \eqref{eq:F}. \rev{Since $u \mapsto \frac12 \int |Du|^2$ is G\^{a}teaux differentiable in $H^1_0$ and convex, we have that for any $v \in \hvdc$,
$$\int_{\Omega} \nabla w_c \cdot (\nabla v-\nabla w_c) + Y \int_\Omega \abs{\nabla v} - Y \int_\Omega \abs{\nabla w_c} -\int_{\Omega_s} f (v-w_c) \gs 0.$$
Using $v = 2w_c$ and $v=0$ (as in \cite[Sections I.3.5.4 and VI.8.2]{DuvLio76}), we obtain
}
\begin{eqnarray*}
\int_\Omega \abs{\nabla w_c}^2 &=& \int_{\Omega_f} \abs{\nabla w_c}^2
= \int_{\Omega_s} w_c - Y \int_{\Omega_f} \abs{\nabla w_c} \\
& \ls &  \int_{\Omega_f} \abs{\nabla w_c} \left[  \sup_{\hvdc}
\frac{\int_{\Omega_s} v}{\int_{\Omega_f} \abs{\nabla v}} - Y \right]
= (Y_c-Y) \int_{\Omega_f} \abs{\nabla w_c}.
\end{eqnarray*}
Thus $w_c \equiv 0$ if $Y\gs
Y_c$.

\begin{ass}\label{ass:linkedparticles}
Even if functions in $\hvdc$ could take different values in different connected components of $\Omega_s$, in what follows we restrict ourselves to functions which are constant in $\Omega_s$. This assumption covers the cases in which $\Omega_s$ is connected (Examples \ref{ex:cylinders}, \ref{ex:SquareSquare}, \ref{ex:posbdy}), when there are two connected components arranged symetrically (Example \ref{ex:bridges}), or when a physical assumption can be made that the particles are linked and have the same possible velocities (Example \ref{ex:tubes}).
\end{ass}

Under assumption \ref{ass:linkedparticles} we set $v = 1$ in $\Omega_s$, and therefore we need to minimize the total variation over the set
\begin{equation}
H_{\diamond,1}^1(\Omega):=
\set{v \in H^1_0(\Omega) :  \int_\Omega v = 0\,, \; v \equiv 1 \text{ in }
\Omega_s}\;.
\label{eq:H1normalized}\end{equation}

It is easy to see that this functional does not necessarily attain a minimum. Hence we use standard relaxation techniques.

\paragraph{Relaxation}
A function $u \in L^1(\R^2)$ is said to be of bounded variation if its distributional gradient $Du$ is a vector valued Radon measure with finite mass, that is
\begin{equation*}
  TV(u) := \abs{Du}(\R^2) =  \sup \left\{\int_\Omega u~ \div z~dx : z \in C^\infty_0(\R^2; \R^2), \|z\|_{L^\infty} \ls 1\right\} < +\infty.
\end{equation*}
The class of
such functions is denoted by $BV(\R^2)$. 
The relaxation of minimizing $TV$ in $H_{\diamond,1}^1(\Omega)$ with respect to strong convergence in $L^1$ \rev{(note that the constraints are preserved)} turns out to be \cite[Prop. 11.3.2]{AttButMic06} minimizing total variation over the set
\begin{equation}\bvdo:=\set{v \in BV(\R^2) : \int_\Omega v = 0\,, \; v \equiv 1 \text{ in } \Omega_s, v \equiv 0 \text{ in } \R^2 \setminus \Omega}\;. \label{eq:BVnormalized}\end{equation}
Since $\bvdo \subseteq \text{BV}(\widetilde{\Omega})$ and $\text{BV}(\widetilde{\Omega}) \subseteq L^1(\widetilde{\Omega})$ with compact embedding (\cite[Cor. 3.49]{AmbFusPal00}) for every bounded $\widetilde{\Omega} \supseteq \Omega$ with $\text{dist}(\partial \Omega, \partial \widetilde{\Omega}) > 0$, the condition $\int_\Omega v = 0$ and compactness in the weak-* topology of $\text{BV}(\widetilde{\Omega})$ (\cite[Thm. 3.23]{AmbFusPal00}) imply that there exists at least one minimizer of $\text{TV}$ in $\bvdo$.
\begin{rem}
 Note that the total variation appearing in the relaxed problem is in $\R^2$, meaning that jumps at the boundary of $\Omega$ are counted. Likewise, in the rest of the paper, every time we speak of total variation with Dirichlet boundary conditions on the boundary of a set $A$, we mean the total variation in $\R^2$ of functions with their values fixed on $\R^2 \setminus A$.
\end{rem}

In the sequel we will repeatedly use the relation between total variation and perimeter of sets. A measurable set $E \subseteq \R^2$ is said to be of finite
perimeter in $\R^2$ if $1_{E}\in BV(\R^2)$, where $1_E$ is the indicatrix (or characteristic function) of the set $E$.
The perimeter of $E$ is defined as $\per{E}:= TV(1_{E}) $.

When $E$ is a set of finite perimeter
with Lipschitz boundary, its perimeter $\per{E}$ coincides with
$\H^1(\partial E)$, where $\H^1$ is the $1$-dimensional Hausdorff measure.
Moreover, we denote the Lebesgue measure of $E$ by $\abs{E}$, so that $\abs{E}:=\int_{\R^2} 1_E$.

We recall the so-called \emph{coarea formula} for $u \in BV(\R^2)$ compactly supported
\begin{equation} TV(u) = \int_{-\infty}^{\infty} \per(u > t) \dd t = \int_{-\infty}^\infty \per( u < t) \dd t, \label{eq:coarea}\end{equation}
as well as the \emph{layer cake formula}, valid for any nonnegative $u \in L^1(\R^2)$
\begin{equation}  \int_{\R^2} u = \int_{0}^\infty |\{u > t \}| \dd t.\label{eq:layercake} 
\end{equation}
For more details on $BV$-functions and finite perimeter sets we refer to \cite{AmbFusPal00}.

Particularly important for our analysis are \emph{Cheeger sets}:
\begin{definition}
 \label{de:ch} (see \cite{Par11})
 Let $\Omega_0$ be a set of finite perimeter. A set $E_0$ minimizing the ratio $$E \mapsto \frac{\per E}{|E|}$$ over
 subsets of $\Omega_0$, is called a \emph{Cheeger set} of $\Omega_0$. The quantity
 $$\lambda = \frac{\per E_0}{|E_0|}$$ is called the \emph{Cheeger constant} of $\Omega_0$.
If $\hat{\Omega}$ is open and bounded, at least one Cheeger set exists \cite[Prop. 3.5, iii)]{LeoPra16}.
Since being a Cheeger set is stable by union \rev{\cite[Prop. 3.5, vi)]{LeoPra16}}, there exists a unique maximal (with respect to $\subset$) Cheeger set.
\end{definition}

\section{Piecewise constant minimizers}
\label{se:pwcm}
We search now for simple minimizers of $TV$ over $\bvdo$. We prove that one can find a minimizer that attains only three values, one of them being zero. After investigation of the particularly simple case where $\Omega_s$ is convex, we tackle the general case in four steps.
\begin{itemize}
 \item Starting from a generic minimizer, in Proposition \ref{prop:cheeger1}, we construct a minimizer whose negative part is constant.
 \item Based on the minimizer with a constant negative part, we then construct a minimizer with constant
 positive part (Theorem \ref{thm:positive}). Thus there exists a minimizer with three different values, a negative one, a positive one (which is constrained to be $1$), and $0$.
 \item We formulate the total variation minimization for three-level functions as a geometrical problem for optimizing the
 characteristic sets of the positive and negative value and study the curvature of the corresponding interfaces.
 \item Finally, we show that we can obtain these optimized characteristic sets by solving two consecutive Cheeger-type problems
 (Theorem \ref{thm:consecutive}).
\end{itemize}

\subsection{Particular case: \texorpdfstring{$\Omega_s$}{The solid} is convex}
\begin{prop}
If $\Omega_s$ is convex, then the function
$$u_0 := 1_{\Omega_s} - \alpha 1_{\Omega_-},$$
where $\Omega_-$ is a Cheeger set of $\Omega \setminus \Omega_s$ and $\alpha = \frac{\abs{\Omega_s}}{\abs{\Omega_-}}$,
is a minimizer of $TV$ in $\bvdo$.
\end{prop}
\begin{proof}
Let $u$ be a minimizer. We write $$u = u^+ - u^-, \text{ with } u^+, u^- \gs 0.$$
Then, we have (by the coarea formula for example)
\begin{equation}
\label{tv+-}
TV(u) = TV(u^+) + TV(u^-).
\end{equation}
Firstly, note that $u \ls 1$: indeed, if $\abs{\set{u>1}} > 0$, then the function
$$\hat{u}:=u\cdot 1_{\{0 < u < 1\}} + 1_{\{u \gs 1\}} - \frac{\int u\cdot 1_{\{0 < u < 1\}} + 1_{\{u \gs 1\}}}{\int u^+} u^-.$$
satisfies $\int \hat{u}=0$ because $\int u^- = \int u^+$, and moreover
\begin{equation*}
 \begin{aligned}
  TV(\hat{u}) &= TV (u\cdot 1_{\{0 < u < 1\}} + 1_{\{u \gs 1\}}) + \frac{\int u\cdot 1_{\{0 < u < 1\}} + 1_{\{u \gs 1\}}}{\int u^+} TV(u^-) \\
  & < TV (u^+) + TV (u^-),
 \end{aligned}
\end{equation*}
which contradicts that $u$ is a minimizer.

Then, let us prove that we can choose $u^+ = 1_{\Omega_s}.$ Thanks to the coarea formula,
$$TV(u^+) = \int_{t=0}^1 \per( u > t) \dd t.$$
Since $u=1$ on $\Omega_s$, for every $0 < t < 1$, we have $\{u\gs t\} \supset \Omega_s$ which implies that $\per(u> t) \gs \per \Omega_s$ by the convexity of $\Omega_s$ (since the projection onto a convex set is a contraction). As a result, we reduce the total variation of $u^+$ by replacing it with $1_{\Omega_s}.$ Replacing then $u^-$ by $\eta u^-$ where $\eta = \frac{|\Omega_s|}{\int u^+} < 1$, we produce a competitor $\tilde u=1_{\Omega_s} - \eta u^-$, which has, since $u$ is a minimizer, the same total variation as $u$.

Now, notice that $\tilde u^-$ minimizes total variation with constraints
\begin{equation*}
 u=0 \text{ on } ( \R^2 \setminus \Omega ) \cup \Omega_s,\quad \int \tilde u^- = |\Omega_s|.
\end{equation*}
We can link this to the Cheeger problem in $\Omega \setminus \Omega_s.$ We denote
$$ \lambda = \min_{E \subset (\Omega \setminus \Omega_s)} \frac{\per E}{|E|}$$
and $E_0$ a minimizer of this ratio. Then, one can write, observing that for $t \ls 0$, $\{\tilde u < t\} \subset (\Omega \setminus \Omega_s)$
\begin{align*}TV(\tilde u^-) = \int_{-\infty}^0 \per (\tilde u < t) \dd t &\gs \lambda \int_{-\infty}^0 |\tilde u < t| \dd t= \lambda \int \tilde u^- \\&=\lambda|\Omega_s|=  \frac{\per E_0}{|E_0|}|\Omega_s| = TV\left( \frac{|\Omega_s|}{|E_0|} 1_{E_0} \right).\end{align*}
Finally, \eqref{tv+-} implies that the function
$$u_0 := 1_{\Omega_s} - \frac{|\Omega_s|}{|E_0|} 1_{E_0} $$
is a minimizer of $TV$ which has the expected form.
\end{proof}

\subsection{General case (\texorpdfstring{$\Omega_s$}{The solid} not convex)}
For any minimizer $u$ on $TV$ in $\bvdo$, there exists a (possibly different) minimizer in which $u^-$ is replaced by a constant function on the characteristic set of the negative part of $u^-$. 
\begin{prop}\label{prop:cheeger1}
Let $\Theta_+ := \supp u^+$. Then,
\begin{equation}\label{eq:u0}
u_0 := u^+ - \frac{\int u^+}{|\Omega_-|} 1_{\Omega_-},
\end{equation}
where $\Omega_-$ is a Cheeger set of $\Omega \setminus \Theta_+$, is a minimizer of $TV$ on $\bvdo$.
In addition, for every $t \ls 0$, the level-sets $\{u<t\}$ are also Cheeger sets of $\Omega \setminus \Theta_+.$
\end{prop}
\begin{proof}
 First, we notice that $u^-$ minimizes $TV$ with constraints $\int u^- = \int u^+$ and $u^- = 0$ on $\Theta_+ \cup (\R^2 \setminus \Omega)$. Let us show that $u^-$ minimizes
$ \frac{TV(v)}{\int v}$
 among all functions supported in $\overline{\Omega\setminus \Theta_+}.$
 Indeed, if we have, for such a $v$, $$\frac{TV(u^-)}{\int u^-} > \frac{TV(v)}{\int v},$$ then $v^- := \frac{\int u^+}{\int v} v$ satisfies $TV(v^-) = \frac{\int u^+}{\int v}  TV (v) < TV(u^-)$, which is a contradiction.
Then, it is well known (see, once again, \cite{Par11}) that the minimizer $v$ can be chosen as an indicatrix of a Cheeger set $\Omega_-$ of $\Omega \setminus \Theta_+$. That shows that $u_0$ is a minimizer.

Now, just introduce $\lambda = \frac{\per \Omega_-}{|\Omega_-|}$ and use the previous computations to write
\begin{align*}
 \lambda \int u^+   & = TV(u^-) = \int_{-\infty}^0 \per(u<t) \dd t= \int_{-\infty}^0 \frac{\per(u<t)}{|u< t|} | u < t| \dd t\\
 &\gs \int_{-\infty}^0 \lambda |u<t| \dd t = \lambda \int u^-.
\end{align*}
Since $\int u^+ = \int u^-$, all these inequalities are equalities and for a.e. $t$, we have $\frac{\per(u<t)}{|u< t|} = \lambda$ and $\{u<t\}$ is therefore a Cheeger set of $\Omega \setminus \Theta_+.$
\end{proof}

In the following, starting from $u_0$, we show that there exists another minimizer of $TV$ if we replace $u_0^+$ by
the indicatrix of a set $\Omega_1$.
\begin{thm}
There exists a minimizer of $TV$ in $\bvdo$ which has the form
\begin{equation}\label{eq:uc}
u_c := 1_{\Omega_1} - \frac{|\Omega_1|}{|\Omega_-|} 1_{\Omega_-},
\end{equation}
where $\Omega_1$ is a minimizer of the functional
\begin{equation}
\mathcal{T}(E):=\per(E) + \frac{\per(\Omega_-)}{|\Omega_-|} |E|
\label{eq:omega1}
\end{equation}
over Borel sets $E$ with $\Omega_s \subset E \subset \Omega \setminus \Omega_-$. In fact, for every $0 \ls t < 1$, the level-sets $E_t:=\{u>t\}$ of every minimizer $u$ minimize $\mathcal{T}$.
\label{thm:positive}
\end{thm}
\begin{proof}
Let $u_0$ be the minimizer of $TV$ in $\bvdo$ from \eqref{eq:u0}. Then
\begin{equation*}
 TV(u_0) = TV(u_0^+) + TV(u_0^-) = TV(u_0^+) + \frac{\per(\Omega_-)}{|\Omega_-|} \int u_0^+
\end{equation*}
Then from \eqref{eq:coarea}, \eqref{eq:layercake}, and \eqref{eq:omega1} it follows:
\begin{equation*}
TV(u_0)=\int_0^1 \per(u_0 > t) + \frac{\per(\Omega_-)}{|\Omega_-|} |u_0 > t| \dd t = \int_0^1 \mathcal{T} (u_0 > t)\dd t \geq \mathcal{T}(\Omega_1).
\end{equation*}
That means, that if we replace $u^+$ by $1_{\Omega_1}$, $TV$ is decreased and thus
$$TV(u_c) \ls TV(u_0) \ls TV(u).$$
Because $u_c$ satisfies $\int u_c = 0$ we see from the last inequality that $u_c$ is a minimizer of $TV$ in $\bvdo$. As before, since $u$ is a minimizer, the inequalities are equalities and we infer the last statement.
\end{proof}

\subsection{Geometrical properties of three-valued minimizers}
We introduce the class
$$ M := \left \{ (E_1,E_-) \subset \Omega \ \middle \vert \ \bigring{E_1} \cap \bigring{E_-} = \emptyset, \ \Omega_s \subset E_1 \right\}.$$
and the functional
$$\mathcal S(E_1,E_-) = \per(E_1) + \frac{|E_1|}{|E_-|} \per(E_-).$$
In addition, for $(E_1,E_-) \in M$ we define the function
\begin{equation*}
 u_c(E_1,E_-) = 1_{E_1} - \frac{|E_1|}{|E_-|} 1_{E_-}\;.
\end{equation*}

\begin{prop}
$\mathcal S$ has a minimizer in $M$. In addition, the second part of every minimizer has positive Lebesgue measure.
\label{prop:existS}
\end{prop}
\begin{proof}
Let $(E_1^n, E_-^n)$ be a minimizing sequence for $\mathcal S$ in $M$. The conditions $\Omega_s \subset E_1$ and $E_- \subset \Omega$ ensure that
$\per(E_1^n) + \per(E_-^n) \ls C,$
so that standard compactness and lower semicontinuity results for sets of finite perimeter \cite{AmbFusPal00} imply existence of a minimizer. \rev{Note that non-empty interiors have positive measure, so the class $M$ is preserved by $L^1$ convergence. Moreover, using the isoperimetric inequality we get
\begin{equation*}
\per(E)\gs \sqrt{4 \pi} |E|^{\frac{1}{2}},\text{ so that }\frac{\per(E)}{|E|}\gs \sqrt{4 \pi} |E|^{-\frac{1}{2}},
\end{equation*}
therefore $|E_-^n|$ is bounded away from zero and the corresponding part of the minimizer has positive measure.
}
\end{proof}
Using Theorem \ref{thm:positive}, we see that the connection between minimizing $TV$ in $\bvdo$ and minimizing $\mathcal S$ is as follows:
\begin{prop}
If the function
$u_c := u_c(\Omega_1,\Omega_-)$ minimizes $TV$ in $\bvdo$, then $(\Omega_1,\Omega_-)$ minimizes $\mathcal S$ in $M$. Conversely, if $(\Omega_1,\Omega_-)$ minimizes $\mathcal S$ in $M$, then $u_c(\Omega_1,\Omega_-)$ minimizes $TV$ in $\bvdo$.
\end{prop}
\begin{remark}
The proposition explains why, in the following, we consider the shape optimization problem of minimizing
$\mathcal S$ in $M$.  \\
We remark that this produces minimizers of $TV$ in $\bvdo$ of a certain (geometric) form, which are not necessarily all of them.
\end{remark}

\begin{figure}
\begin{center}
\includegraphics[scale=0.4]{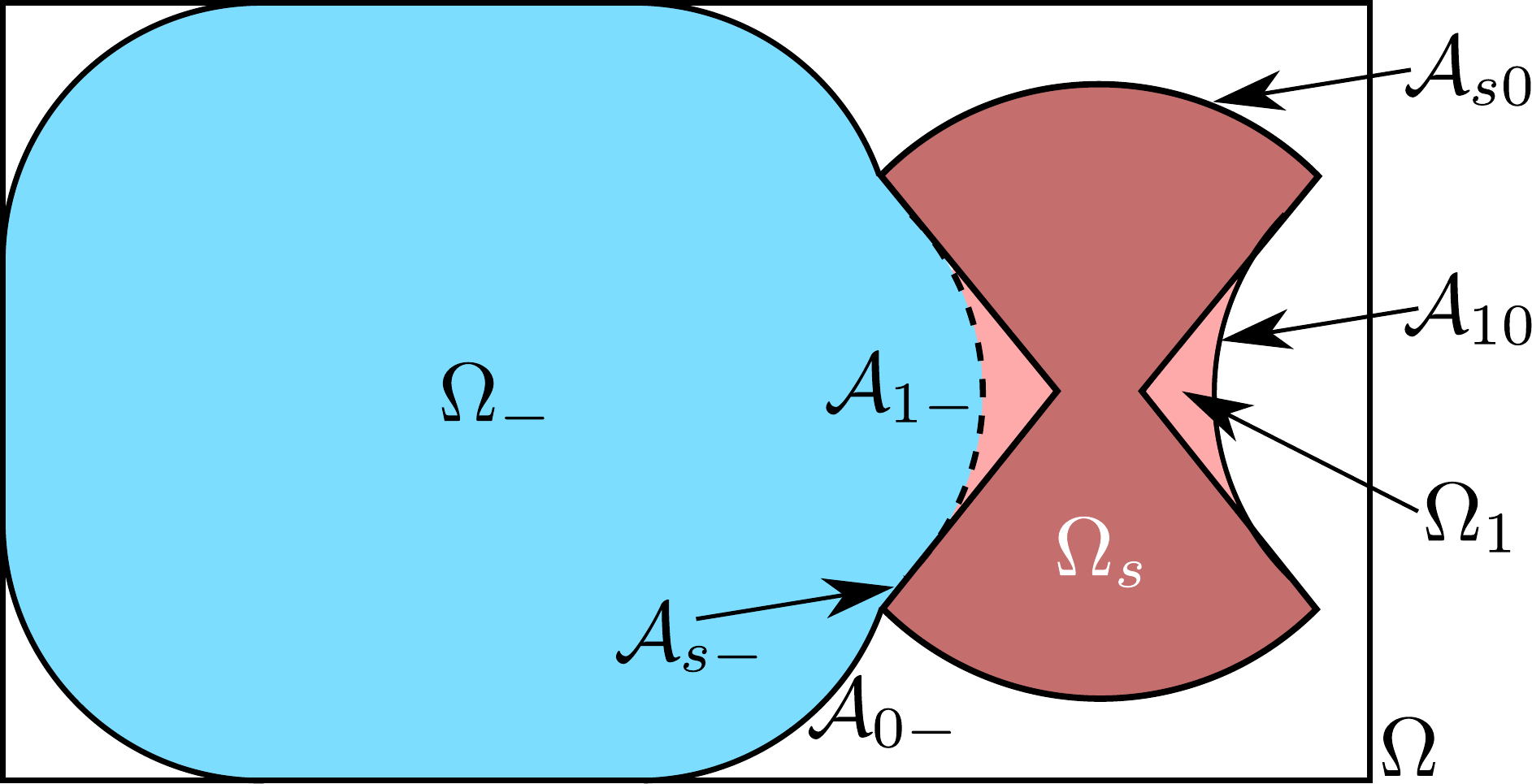}
\end{center}
\caption{Interfaces present in minimizers of $\mathcal S$.}
\label{fig:interfaces}
\end{figure}

In what follows, we consider small perturbations of a minimizer $(\Omega_1,\Omega_-)$ of $\mathcal S$ in which only one of the sets is changed. This will be enough to determine the curvature of their boundaries, which we split as follows (see Figure \ref{fig:interfaces})
 \begin{equation*}
 \begin{array}{c}
  \mathcal{A}_{1-} = \set{ x \in \Omega: x \in \partial \Omega_1, x \in \partial \Omega_-}, \quad
  \mathcal{A}_{10} = \set{ x \in \Omega: x \in \partial \Omega_1, x \notin \partial \Omega_-},\\
  \mathcal{A}_{0-} = \set{ x \in \Omega: x \notin \partial \Omega_1, x \in \partial \Omega_-}, \quad
  \mathcal{A}_{s-} = \set{ x \in \Omega: x \in \partial \Omega_s, x \in \partial \Omega_-},\\
  \mathcal{A}_{s0} = \set{ x \in \Omega: x \in \partial \Omega_s, x \notin \partial \Omega_-}\;.
  \end{array}
\end{equation*}
We denote by $\kappa_1, \kappa_-$ the curvature functions of $\Omega_1, \Omega_-$, defined in $\partial \Omega_1,\partial
\Omega_-$ through their outer normals $n_1,n_-$ (i.e.~a circle has positive curvature).

For a generic set of finite perimeter in $\R^2$ only a distributional curvature is available \cite[Rem. 17.7]{Mag12}. However, since
$\Omega_1$ and $\Omega_-$ minimize the functionals $\mathcal{S}(\cdot, \Omega_-)$ and $\mathcal{S}(\Omega_1, \cdot)$ respectively,
regularity theorems for $\Lambda$-minimizers of the perimeter \cite[Thm. 26.3]{Mag12} are applicable to them. In consequence,
$\mathcal{A}_{1-}$, $\mathcal{A}_{0-}$ and $\mathcal{A}_{10} \setminus \mathcal{A}_{s0}$, are locally graphs of $C^{1,\gamma}$ functions.
Combined with standard regularity theory for uniformly elliptic equations \cite{GilTru01}, one obtains higher regularity, so that, in
particular, the curvatures $\kappa_1, \kappa_-$ are defined classically on those interfaces (on $\partial \Omega_s \cap \partial \Omega_1$,
no information is provided).

\begin{prop}
Let $(\rev{\Omega_1},\Omega_-)$ be a minimizer of $\mathcal{S}$. Then, the curvatures $\kappa_-$, $\kappa_1$
of the interfaces  $\mathcal A_{0-}$ and $\mathcal{A}_{10} \setminus \mathcal{A}_{s0}$ are given by
$$\kappa_- = \frac{\per \Omega_-}{|\Omega_-|} \text{ on } \mathcal A_{0-}
  \text{ and } \kappa_1 = - \frac{\per \Omega_-}{|\Omega_-|}\text{ on }\mathcal{A}_{10} \setminus \mathcal{A}_{s0}.$$
In consequence, $\mathcal{A}_{0-}$ and $\mathcal{A}_{10} \setminus \mathcal{A}_{s0}$ are composed of pieces of circles of
radius $\frac{|\Omega_-|}{\per \Omega_-}$.
 \label{prop:sameradius1}
\end{prop}
\begin{proof}
For every $x \in \mathcal{A}_{10} \setminus \mathcal{A}_{s0}$ we consider a perturbed domain $\Omega_1^w$ (see Figure \ref{fig:interfaces}),
such that $\Omega_1^w = (I + \overrightarrow{w}) (\Omega_1)$, where $\overrightarrow{w}$ is supported in a
neighborhood of $x$. Calling $w := \overrightarrow{w} \cdot n_1$ and thanks to the first variation formula
\cite[Thm.~17.5 and Rem.~17.6]{Mag12} we can develop the first variation of $\mathcal{S}(\cdot, \Omega_-)$ at a minimizer
$\Omega_1$ in direction $w$ and obtain
$$ \int_{\mathcal{A}_{10} \setminus \mathcal{A}_{s0}} \kappa_1 w + w \frac{\per(\Omega_-)}{|\Omega_-|} \dd \mathcal H^{1}=0.$$
Since $w$ was arbitrary, we get the optimality condition for $\Omega_1$:
$$\kappa_1 + \frac{\per(\Omega_-)}{|\Omega_-|} = 0 \quad \text{ in } \mathcal{A}_{10} \setminus \mathcal{A}_{s0}.$$
Proceeding similarly for $\Omega_-$ we obtain
$$\frac{1}{|\Omega_1|} \left( \frac{\kappa_-}{|\Omega_-|} - \frac{\per(\Omega_-)}{|\Omega_-|^2} \right)=0
\quad \text{ in } \mathcal{A}_{0-}.$$
This shows that the curvatures of $\mathcal A_{1-} \setminus \mathcal{A}_{s-}$ and
$\mathcal{A}_{1-} \setminus \mathcal{A}_{s-}$ are constant with values $\kappa_1=-\kappa_-=\frac{\per(\Omega_-)}{|\Omega_-|}$.
This in particular shows that these interfaces are composed of circles of radii
$\frac{|\Omega_-|}{\per(\Omega_-)}$.
\end{proof}

\begin{prop}
Let $(\rev{\Omega_1},\Omega_-)$ be a minimizer of $\mathcal{S}$. Then
$$ \kappa_- = \frac{\per \Omega_-}{|\Omega_-|} = - \kappa_1 \text{ on } \mathcal{A}_{1-} \setminus \mathcal{A}_{s-}.$$
Thus, $\mathcal{A}_{1-} \setminus \mathcal{A}_{s-}$ consists of pieces of circle with the same radius of
Proposition \ref{prop:sameradius1}.
 \label{prop:sameradius2}
\end{prop}
\begin{proof}
First, we note that since $\mathcal A_{1-} \setminus \mathcal A_{s-} \subset \partial \Omega_1 \cap \partial \Omega_-$, we must have
$$ \kappa_1 = - \kappa_-  \text{ on } \mathcal A_{1-} \setminus \mathcal A_{s-}.$$

Now, we perturb $\Omega_1$ while keeping $\Omega_-$ fixed. In this context, $\Omega_1$ is a minimizer of $E \mapsto \mathcal{S}(E, \Omega_-)$ with constraints $E \subset \Omega$ and $\Omega_1 \cap E = \emptyset$. Since $\Omega_-$ is fixed the second constraint allows only inward perturbations. We therefore perturb $\Omega_1$ in its exterior normal direction with a function $w \ls 0$ supported in $\mathcal{A}_{1-} \setminus \mathcal{A}_{s-}$.  The variation formula for $\Omega_1$ in direction $w$ provides
$$ \int_{\mathcal A_{1-} \setminus \mathcal A_{s-}} \kappa_1 w + \int w \frac{\per(\Omega_-)}{|\Omega_-|} \dd \mathcal H^{1} \gs 0,$$
which yields
$$\kappa_1 \ls - \frac{\per(\Omega_-)}{|\Omega_-|} \text{ on } \mathcal A_{1-} \setminus \mathcal A_{s-}.$$

Now, we fix $\Omega_1$ and perturb $\Omega_-$ similarly with $w \ls 0$, again supported in $\mathcal{A}_{1-} \setminus \mathcal{A}_{s-}$ (so the perturbation goes inside $\Omega_-$). Since $\Omega_-$ now minimizes $\mathcal{S}(\Omega_1, \cdot)$, we get
$$ \int_{\mathcal A_{1-} \setminus \mathcal A_{s-}}  w \kappa_- \frac{|\Omega_1|}{|\Omega_-|} -  w \frac{|\Omega_1|}{|\Omega_-|^2}\per(\Omega_-) \dd \mathcal H^1 \gs 0,$$
which gives
$$\kappa_- \ls \frac{\per(\Omega_-)}{|\Omega_-|} \text{ on } \mathcal{A}_{1-} \setminus \mathcal{A}_{s-}.$$
\end{proof}

\begin{prop}
Let $E$ be a connected component of $\Omega \setminus ( \Omega_- \cup \Omega_1)$ such that $\partial E \cap \partial \Omega = \emptyset$. Then, $(\Omega_1 \cup E, \Omega_-)$ and $(\Omega_1, \Omega_- \cup E)$ belong to $M$ and minimize $\mathcal S$.
\label{prop:nozero}
\end{prop}
\begin{proof}
We abbreviate $\lambda = \frac{\per \Omega_-}{|\Omega_-|}.$ Then because
$E \cap \Omega_- = E \cap \Omega_1= \emptyset$, the pairs $(\Omega_1 \cup E, \Omega_-)$ and $(\Omega_1, \Omega_- \cup E)$ both belong to $M$ \rev{and} we have
$$\per(\Omega_1 \cup E) + \lambda |\Omega_1 \cup E| \gs \per(\Omega_1)+\lambda|\Omega_1|,$$
which implies because $E \cap \Omega_1= \emptyset$
\begin{equation}
 \label{eq:e1}
 \lambda |E| \gs \per(\Omega_1)-\per(\Omega_1 \cup E).
\end{equation}
Because $\Omega_-$ is a Cheeger set of $\Omega \setminus \Omega_1$, we have
$$\frac{\per(\Omega_- \cup E)}{|\Omega_- \cup E|} \gs \frac{\per(\Omega_-)}{|\Omega_-|}$$
which, because $E \cap \Omega_- = \emptyset$, implies
$$ \per(\Omega_- \cup E) |\Omega_-| \gs \per(\Omega_-) (|\Omega_-|+|E|)\,,$$
which \rev{yields}
\begin{equation}
\label{eq:e2}
 \per(\Omega_- \cup E) - \per(\Omega_-) \gs \lambda |E|.
\end{equation}
In summary, we have shown in \eqref{eq:e1} and \eqref{eq:e2} that
$$\per(\Omega_- \cup E) - \per(\Omega_-) \gs \lambda |E| \gs \per(\Omega_1)-\per(\Omega_1 \cup E).$$
Since $\partial E \cap \partial \Omega = \emptyset$ and $E \cap \Omega_- = E \cap \Omega_1= \emptyset$, we know $\partial E \subset \partial \Omega_1 \cup \partial \Omega_-$. Furthermore, \rev{$E \cap \Omega_- = E \cap \Omega_1= \emptyset$ also implies that the common boundaries between $E$ and $\Omega_-$, and between $E$ and $\Omega_1$ have opposite-pointing outer normals} and
one can write \cite[Thm. 16.3]{Mag12}
$$\per(\Omega_- \cup E) - \per(\Omega_-) =  \per(\Omega_1)-\per(\Omega_1 \cup E)$$ which implies that all the inequalities above are equalities, and the set $E$ can be joined to $\Omega_-$ or $\Omega_1$ without changing the value of $\mathcal S$.
\end{proof}

In the following we show that one may obtain minimizers of $\mathcal S$ (and therefore minimizers of $TV$ in $\bvdo$ with three values) in two simpler steps:
\begin{enumerate}
 \item Solve the Cheeger problem for $\Omega \setminus \Omega_s$. Let $\Omega_c$ be the maximal Cheeger set and $\lambda_c := \frac{\per \Omega_c}{|\Omega_c|}$ its Cheeger constant.
 \item Obtain the minimal (with respect to $\subset$) minimizer $\Omega_{1c}$ of
$$ \per(E) + \lambda_c |E| \text{ over } \rev{E \supset \Omega_s}.$$
\end{enumerate}
Note that minimizers of the second problem exist by an argument similar to Proposition \ref{prop:existS}.
\begin{thm}
The pair $(\Omega_{1c},\Omega_c)$ minimizes $\mathcal S$.
\label{thm:consecutive}
\end{thm}
\begin{proof}
Let $\lambda := \frac{\per \Omega_-}{|\Omega_-|}$ (by definition of the Cheeger set $\Omega_c$,
we have $\lambda \gs \lambda_c$).
Let also $E$ be the smallest (with respect to $\subset$) minimizer of
\begin{equation} \hat E \mapsto \per(\hat E) + \lambda |\hat E| \text{ subject to } \Omega_s \subset \hat E. \label{eq:pbE}\end{equation}

We want to show that $E \cap \Omega_- = \emptyset$, that is $E$ is also a minimizer of $\per(\cdot) + \lambda |\cdot|$ with respect to the constraints $E \cap \Omega_- = \emptyset$ and $\Omega_s \subset E$.

Because $E \setminus \Omega_-$ is admissible in \eqref{eq:pbE},
$$\per(E\setminus \Omega_-) + \lambda |E \setminus \Omega_-| \gs \per(E) + \lambda |E|.$$
On the other hand, $\Omega_-$, as a Cheeger set of $\Omega \setminus \Omega_1$, is a minimizer of
\begin{equation} \hat{E} \rev{\mapsto} \per(\hat{E}) - \lambda |\hat{E}| \text{ subject to } \hat{E} \cap \Omega_1 = \emptyset. \label{eq:pbOm} \end{equation}
Then $\Omega_- \setminus E$ is a competitor for \eqref{eq:pbOm},
$$ \per(\Omega_- \setminus E) - \lambda |\Omega_- \setminus E| \gs \per(\Omega_-) - \lambda |\Omega_-|. $$
Summing these two inequalities and using that (see \cite[Exercise 16.5]{Mag12})
$$\per(E\setminus \Omega_-) + \per(\Omega_- \setminus E) \ls \per(E) + \per(\Omega_-),$$
we obtain
$$ \lambda \left( |E \setminus \Omega_-| - |\Omega_- \setminus E| \right) \gs  \lambda \left( |E| - |\Omega_-| \right).$$
Since this last inequality is an equality, it is also true for the two previous ones, and we can conclude that
$$\per(E\setminus \Omega_-) + \lambda |E \setminus \Omega_-| = \per(E) + \lambda |E|$$
which implies, since $E$ is minimal with respect to the inclusion, that $E \cap \Omega_- = \emptyset$.

Similarly, if $E_c$ is a minimizer of
\begin{equation}\hat E \mapsto \per \hat E + \lambda_c |\hat E| \text{ with constraint } \rev{\Omega_s} \subset \rev{\hat E},\label{eq:pbEc}\end{equation}
one can prove that $E_c \cap \Omega_c = \emptyset$.

We have proved that $\Omega_1,\Omega_{1c}$ minimize $\per (\cdot) + \lambda\abs{\cdot}, \ \per(\cdot)+\lambda_c \abs{\cdot}$ with the same constraint (containing $\Omega_s$). Hence, $\Omega_1 \cap \Omega_{1c}$ is admissible in \eqref{eq:pbE} and $\Omega_1 \cup \Omega_{1c}$ is admissible for \eqref{eq:pbEc}, which implies
$$ \per(\Omega_1 \cap \Omega_{1c})+ \lambda |\Omega_1 \cap \Omega_{1c}| \gs \per \Omega_1 + \lambda |\Omega_1|,$$
$$ \per(\Omega_1 \cup \Omega_{1c}) + \lambda_c |\Omega_1 \cup \Omega_{1c}| \gs \per \Omega_{1c} + \lambda_c |\Omega_{1c}|.$$
Summing these inequalities and recalling that \cite[Lem. 12.22]{Mag12} 
$$\rev{\per(\Omega_1 \cap \Omega_{1c}) + \per(\Omega_1 \cup  \Omega_{1c}) \ls \per(\Omega_1) + \per(\Omega_{1c}),}$$
\rev{we get}
$$\lambda_c |\Omega_1 \setminus \Omega_{1c} | \gs \lambda |\Omega_1 \setminus \Omega_{1c} |.$$
Then, if $\lambda_c < \lambda$ we obtain $\Omega_{1c} \supset \Omega_1$ and if $\lambda = \lambda_c$, all the inequalities above are equalities, which implies once again (using the minimality of $\Omega_1$) that $\Omega_{1c} \supset \Omega_1.$
Then, $\Omega_c \cap \Omega_1 = \emptyset$ hence $\Omega_c$ is also a Cheeger set of $\Omega \setminus \Omega_1.$
\end{proof}

\begin{rem}
By the statements in the previous section about level sets of the generic minimizer $u$, we infer that the only lack of uniqueness present in the minimization of $TV$ in $\bvdo$ is that of the corresponding geometric problems. More precisely, if the Cheeger set of $\Omega \setminus \Omega_s$ is unique, (which is shown in \cite[Thm. 1]{CasChaNov10} to be a generic situation), then the minimizer of $TV$ in $\bvdo$ is unique as well. Indeed, with the same arguments as in the proof of Proposition \ref{prop:nozero}, one sees that the minimizer of \eqref{eq:omega1} is also unique, which implies by Proposition \ref{prop:cheeger1} and Theorem \ref{thm:positive} that the level-sets of $u$ are all uniquely determined.
\end{rem}

\subsection{Behavior of \texorpdfstring{$Y_c$}{the yield number} as \texorpdfstring{$\Omega$}{the domain} grows large}
\begin{prop}\label{prop:growingdomain}
Let $\Omega_0$ be a convex set and $\Omega_s \subset \Omega_0$, both \rev{containing} the origin, and assume that $|\Omega_s| = 1$. For $\alpha \gs 1$, let $\Omega = \alpha \Omega_0$, i.e.~we consider the domain to be a rescaling of $\Omega_0$ (note that $\Omega_s \subset \alpha \Omega_0$). Then 
$$\lim_{\alpha \to \infty} Y_c(\alpha) = \frac{1}{\min\limits_{E \supset \Omega_s}{\per E}}.$$
\end{prop}
\begin{proof}
We recall that
$$Y_c(\alpha) = \frac{|\Omega_s|}{\inf_{M_\alpha} \mathcal S},$$
where
$$M_\alpha := \left \{ (E_1,E_-) \subset \alpha\Omega_0 \ \vert \ \bigring{E_1} \cap \bigring{E_-} = \emptyset, \ \Omega_s \subset E_1 \right\}.$$
Then, noticing that for every $\tilde{\Omega}$ such that $\Omega_s \subset \tilde{\Omega} \subset \alpha\Omega_0$ we have $(\tilde{\Omega}, \alpha \Omega_0 \setminus \tilde{\Omega}) \in M_\alpha$, one can write
\begin{align*}\inf_{M_\alpha} \mathcal S &\ls \mathcal S\left(\tilde{\Omega}, \alpha \Omega_0 \setminus \tilde{\Omega}\right)=\per(\tilde{\Omega}) + \frac{\per(\alpha \Omega_0)+\per(\tilde{\Omega})}{|\alpha \Omega_0| - \rev{|\tilde \Omega|}}|\tilde{\Omega}| \\
 & \ls \per(\tilde{\Omega}) + \frac{\alpha\per( \Omega_0)+\per(\tilde{\Omega})}{\alpha^2|\Omega_0| - \rev{|\tilde \Omega|}}|\tilde{\Omega}| \xrightarrow[\alpha \to \infty]{}\per(\tilde{\Omega}).
\end{align*}
On the other hand, \rev{since $(\tilde{\Omega}, \alpha \Omega_0 \setminus \tilde{\Omega}) \in M_\alpha$},
$$\mathcal S(\tilde{\Omega}, \alpha \Omega_0 \setminus \tilde{\Omega}) \gs \per(\tilde{\Omega}).$$
Optimizing in $\tilde{\Omega}$ establishes the result.
\end{proof}

\begin{rem}
\label{rem:moreparticles}
If $\Omega_s$ is indecomposable (i.e., `connected' in an adequate sense for this framework), we have by \cite[Prop. 5]{FerFus09} that $$\min\limits_{E \supset \Omega_s}{\per E} = \per(\operatorname{Co}(\Omega_s)),$$ where $\operatorname{Co}(X)$ is the convex envelope of $X$.
\end{rem}

\begin{rem}
\label{rem:nocriticalsize}
As may be seen in examples \ref{ex:cylinders} and \ref{ex:SquareSquare}, the above limit is not attained at a finite $\alpha$. There is no `critical size' at which the boundary of $\Omega$ stops playing a role. We see that the limiting $Y_c$ is approached at least as $O(1/\alpha)$ as $\alpha \to \infty$.
\end{rem}


\section{Application examples}\label{se:examples}
In the previous section, we have seen that 
the free boundaries of the optimal sets are composed of pieces of circles of the same radius, which suggests that one might be able to use morphological operations to construct these minimizers. We introduce these now.

\begin{definition}[Opening, Closing]
For a set $X$ and $r>0$,
We define the \textbf{opening} of $X$ with radius $r$ by
\begin{equation*}
 \open{r}{X}:=\bigcup_{x: B_{r}(x) \subset X} B_{r}(x)\;,
\end{equation*}
where $B_r(x)$ is the disk with radius $r$ and center $x$.
Additionally we define the \textbf{closing} of $X$ with radius $r$ as
$$\close{{r}}{X}:= \R^2 \setminus \lr{\open{{r}}{\R^2 \setminus X}}.$$
\end{definition}

\subsection{Morphological operations and Cheeger sets}
The Cheeger problem is far from being entirely understood. Nonetheless, it is for convex sets. As a result, if $\Omega$ is convex and $\Omega_s = \emptyset$, the Cheeger set $\Omega_-$ of $\Omega$ satisfies
\begin{itemize}
 \item $\Omega_-$ is unique,
 \item $\Omega_-$ is convex and $\mathcal C^{1,1}$,
 \item $\Omega_- = \open{r}{\Omega}$ where $r$ is the Cheeger constant of $\Omega$.
\end{itemize}
In the general case, for a Cheeger set $\Omega_-$ of $\Omega \setminus \Omega_s$, few results are available \cite{LeoPra16}
\begin{itemize}
 \item The boundaries of $\Omega_-$ are pieces of circles of radius $\frac 1 \lambda$ ($\lambda$ is the Cheeger constant of $\Omega \setminus \Omega_s$) which are shorter than half the corresponding circle.
 \item If $x_0$ is a smooth point of $\partial (\Omega \setminus \Omega_s)$ and belongs to $\partial \Omega_-$, then $\partial \Omega_-$ is $\mathcal C^{1,1}$ around $x_0$ \cite[Th.~2]{CasChaNov10}.
 \item We also have \cite[Lem. 2.14]{LeoPra16}, which basically tells that if the maximal Cheeger set of $\Omega \setminus \Omega_s$ contains a ball of radius $\frac 1\lambda$, then it also contains all the balls of radius $\frac 1\lambda$ obtained by rolling the first ball inside $\Omega \setminus \Omega_s.$ \\
\end{itemize}

\begin{rem}
\label{rem:distanceboundary}
Let $\Omega$ and $\Omega_s$ be convex and let $\lambda$ be the Cheeger constant of $\Omega.$ If $d(\Omega_s,\partial \Omega) \gs \frac{2}{\lambda}$, then the maximal Cheeger set of $\Omega \setminus \Omega_s$ can be obtained rolling a ball of radius $\frac{1}{\lambda_0} < \frac{1}{\lambda}$ around $\Omega_s$ ($\lambda_0 \gs \lambda$ being the Cheeger constant of $\Omega \setminus \Omega_s$). In particular, it fills a neighborhood of $\partial \Omega_s$ in $\Omega \setminus \Omega_s$.
\end{rem}

\subsection{Single convex particles}
We start with two simple examples in which a single convex particle is placed centrally within a larger convex domain.

\begin{example}\label{ex:cylinders}
[Circular $\Omega$] \\
\medskip
\noindent
a) Let $\Omega_s,\Omega$ be two circles with radii $\frac{1}{\sqrt{\pi}}$, $R$, ensuring that $|\Omega_s| = 1$.
Since in this case $\open{r}{\Omega} = \Omega$ for all $r \ls R$,
$\Omega_- = \Omega\setminus \Omega_s$ minimizes $S(\Omega_s, \cdot)$.

Thus, $\Omega_c = \Omega\setminus \Omega_s$ and $\Omega_{1c} = \Omega_s$. We have \[ \lambda_c = \frac{\per \Omega_c}{|\Omega_c|} = \frac{2\pi R + 2 \sqrt{\pi}}{\pi R^2 - 1}, \]
and
\[ Y_c = \frac{\rev{|\Omega_{1c}|}}{Per(\Omega_{1c}) + \lambda_c |\Omega_{1c}|} = \frac{1}{2 \sqrt{\pi} + \frac{2\pi R + 2 \sqrt{\pi}}{\pi R^2 - 1}}  . \]

We may also construct the minimizer of $\text{TV}$ over $\bvdo$, given (in cylindrical coordinates) by $v_0: [0,\infty] \times [0,\pi]\rightarrow \R:$
\begin{equation*}
v_0(r,\phi):=
\begin{cases}
\abs{\Omega_s} = 1 & \text{ for }\; 0\ls r \ls \frac{1}{\sqrt{\pi}} \,,\\
-\frac{\abs{\Omega_s}}{\abs{\Omega}-\abs{\Omega_s}}
= -  \frac{1}{R^2 \pi - 1}
&
\text{ for }\; \frac{1}{\sqrt{\pi}} < r \ls R\,,\\
0 & \text{ for }\; R < r < \infty\,
\end{cases}
\end{equation*}
(evidently axisymmetric). The total variation is:
\begin{align*}
\abs{D v_0}(\Omega) &=
\per{\Omega_s}  +
\lr{\per{\Omega_s}+\per{\Omega} }
\frac{\abs{\Omega_s}}{\abs{\Omega}-\abs{\Omega_s}}  \\
&=
2\sqrt{\pi} +
\frac{2\sqrt{\pi} + 2R\pi}{R^2 \pi - 1}  = \frac{1}{Y_c}\;.
\end{align*}
For $R \to \infty$ the limit is $\per{\Omega_s} = 2 \sqrt{\pi}$
and $Y_c$ approaches $\frac{1}{2\sqrt{\pi}}$.

\medskip
\noindent
b) As a slight variation on the above now let $\Omega_s$ be the unit square. Again we find  $\Omega_c = \Omega\setminus \Omega_s$ and $\Omega_{1c} = \Omega_s$, and hence
\[ \lambda_c = \frac{\per \Omega_c}{|\Omega_c|} = \frac{2\pi R + 4}{\pi R^2 - 1}, \]
and
\[ Y_c = \frac{1}{Per(\Omega_{1c}) + \lambda_c |\Omega_{1c}|} = \frac{1}{4 + \frac{2\pi R + 4}{\pi R^2 - 1}}  \to 0.25 \mbox{ as } R \to \infty. \]
\end{example}

\begin{example}\label{ex:SquareSquare}
[Square $\Omega$] \\
We now consider $\Omega$ to be a square of side $L$.
In the absence of $\Omega_s$ the optimal set $\Omega_-$ is given by $\open{r_{\infty}}{\Omega}$ for $r_{\infty} = L/(2+\sqrt{\pi}) = 1/\lambda_c$; see \cite{MosMia65}.

\medskip
\noindent
a) Now consider a centrally positioned unit square $\Omega_s$, within $\Omega$ of side $L > 1$. The optimal set $\Omega_-$ is given by $\open{r}{\Omega}\setminus \Omega_s$ for some $r>0$. We have
$\abs{\open{r}{\Omega}} = \abs{\Omega} +r^2 \lr{\pi -  4}$,
$\per{\open{r}{\Omega}} = \per{\Omega}+ r \lr{ 2\pi - 8}$, and to find $r = r(L)$ we use Propositions \ref{prop:sameradius1} and \ref{prop:sameradius2}:
\begin{equation*}
\frac{1}{r} = \frac{\per(\open{r}{\Omega}\setminus \Omega_s)}{\abs{\open{r}{\Omega}\setminus \Omega_s}}
 = \frac{4L+4+2r\lr{\pi-4}} {L^2-1+r^2\lr{\pi-4}}.
\end{equation*}
The resulting quadratic equation gives the optimal $r(L)$:
\[ r(L) = \frac{L}{2}\frac{1+1/L}{1-\pi/4}\left( 1 - \sqrt{1- (1-\pi/4)\frac{1-1/L}{(1+1/L)} }\right)  . \]
We find that $r(L) < r_{\infty}$ with $r(L) \to  r_{\infty}$ as $L \to \infty$ and $r(L) \to  0$ as $L \to 1^+$, as expected. Consequently, $\Omega_{c} = \open{r(L)}{\Omega}\setminus \Omega_s$ and
the Cheeger constant $\lambda_c(L)$ is:
\[ \lambda_c(L) = \frac{\per(\open{r(L)}{\Omega}\setminus \Omega_s)}{\abs{\open{r(L)}{\Omega}\setminus \Omega_s}}
 = \frac{4L+4+2r(L)\lr{\pi-4}} {L^2-1+r(L)^2\lr{\pi-4}}. \]
Again we have $\Omega_{1c} = \Omega_s$, and
\[ Y_c(L) = \frac{1}{Per(\Omega_{1c}) + \lambda_c(L) |\Omega_{1c}|} =
\frac{1}{4 + \lambda_c(L)}  .\]

The minimizer of $\text{TV}$ over $\bvdo$ is constructed from the optimal sets:
$$u_{r(L)}:=
1_{\Omega_s}-\frac{\abs{\Omega_s}}{\abs{\open{r(L)}{\Omega}}-\abs{\Omega_s}}
1_{\open{r(L)}{\Omega} \setminus \Omega_s} \;$$
with total variation:
\begin{align*}
\abs{D u_{r(L)}}(\Omega) &= \per{\Omega_s}  +
\frac{\lr{\per{\Omega_s} + \per{\Omega}+ r(L) \lr{ 2\pi - 8}} \abs{\Omega_s}}{
\abs{\Omega} +r(L)^2 \lr{\pi -  4}    - \abs{\Omega_s}}\\ &=
4  +
\frac{\lr{4 + 4L+ r(L) \lr{ 2\pi - 8}} }{L^2 +r(L)^2 \lr{\pi -  4}    - 1}
\end{align*}

\medskip
\noindent
b) We replace $\Omega_s$ by circle of radius $1/\sqrt{\pi}$, ensuring $\abs{\Omega_s} = 1$, and consider $L>2/\sqrt{\pi}$. The calculations are similar. Again the optimal set $\Omega_-$ is $\open{r}{\Omega}\setminus \Omega_s$ with $r=r(L)$ determined from Propositions \ref{prop:sameradius1} and \ref{prop:sameradius2}. We now find:
\[ r(L) = \frac{L}{2}\frac{1+\sqrt(\pi)/(2L)}{1-\pi/4}\left( 1 - \sqrt{1- (1-\pi/4)\frac{1-1/L^2}{(1+\sqrt(\pi)/(2L))^2} }\right)  . \]
Thus, $\Omega_{c} = \open{r(L)}{\Omega}\setminus \Omega_s$, $\Omega_{1c} = \Omega_s$, and
\[ \lambda_c(L) = \frac{\per(\open{r(L)}{\Omega}\setminus \Omega_s)}{\abs{\open{r(L)}{\Omega}\setminus \Omega_s}}
 = \frac{4L+2\sqrt{\pi} +2r(L)\lr{\pi-4}} {L^2-1+r(L)^2\lr{\pi-4}}. \]
\[ Y_c(L) = \frac{1}{Per(\Omega_{1c}) + \lambda_c(L) |\Omega_{1c}|} =
\frac{1}{2\sqrt{\pi} + \lambda_c(L)}  .\]
\end{example}

Figure \ref{fig:squarepeg1}a plots the results of example \ref{ex:SquareSquare} at different $L$. Interestingly, although $\lambda_c(L)$ is smaller for the circular $\Omega_s$, it is only very marginally so. Figure \ref{fig:squarepeg1}b plots the yield limit $Y_c(L)$ for both $\Omega_s$. Here we see a significant difference: the circular $\Omega_s$ requires a larger yield stress to prevent motion. As we have seen that $\lambda_c(L)$ is similar for both $\Omega_s$, this difference in $Y_c$ stems almost entirely from $Per(\Omega_{1c}) = Per(\Omega_{s})$ (in these examples). We may deduce from the expressions derived that $\lambda_c(L) \sim O(1/L)$ as $L \to \infty$ and hence that $Y_c(L) \to 1/Per(\Omega_{s}) + O(1/L)$ as $L \to \infty$; see also Proposition \ref{prop:growingdomain}. The same behaviours are observed with the earlier example \ref{ex:cylinders}, in a circle of radius $R$, i.e.~little difference in $\lambda_c(R)$, significant difference in $Y_c(R)$, stemming primarily from $Per(\Omega_{s})$, and similar asymptotic trends as $R \to \infty$.

\begin{figure}
\begin{center}
\includegraphics[width = 5cm]{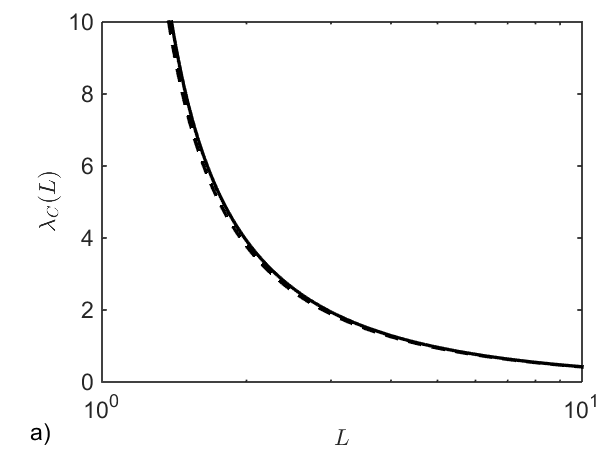}
\includegraphics[width = 5cm]{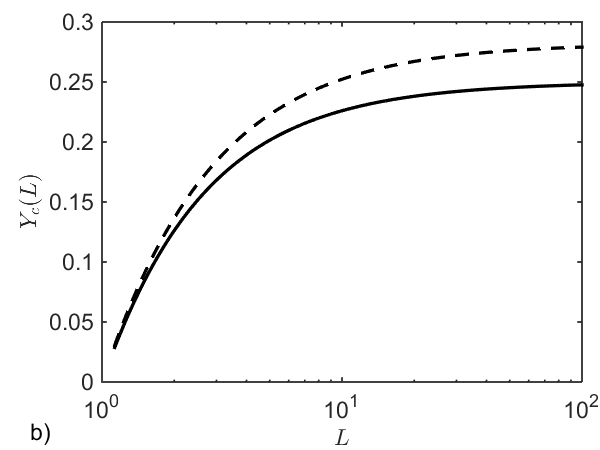}
\end{center}
\caption{Comparison of results of example \ref{ex:SquareSquare} at different $L$: a) $\lambda_c(L)$; b) $Y_c(L)$. Circular $\Omega_s$ is marked with the broken line and square $\Omega_s$ is marked with the solid line.}
\label{fig:squarepeg1}
\end{figure}

We might also seek to compare examples \ref{ex:cylinders} and \ref{ex:SquareSquare} directly. The scaling introduced ensures $\abs{\Omega_s} = 1$, matching the buoyancy force felt by each particle. By setting $L^2 = \pi R^2$ we also match the area of fluid within $\Omega \setminus \Omega_s$. Figure \ref{fig:squarepeg2}a plots $\lambda_c(R)$ and $\lambda_c(L(R))$. Figure \ref{fig:squarepeg2}b plots $Y_c(R)$ and $Y_c(L(R))$. We observe that $\lambda_c(R) < \lambda_c(L(R))$, for the same $\Omega_s$, but again the effect is marginal and $\lambda_c$ is very close for all 4 cases. Interestingly, in Figure \ref{fig:squarepeg2}b we see that by scaling $L^2 = \pi R^2$ the effects of the shape of $\Omega$ are minimized: $Y_c(R)$ and $Y_c(L(R))$ are very close for the same $\Omega_s$, whether it be circular or square.

\begin{figure}
\begin{center}
\includegraphics[width = 5cm]{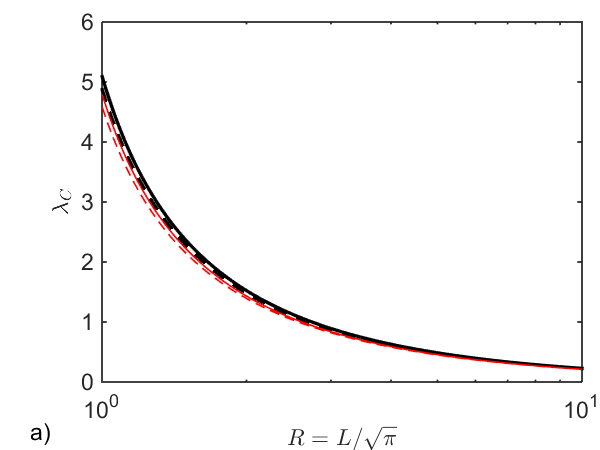}
\includegraphics[width = 5cm]{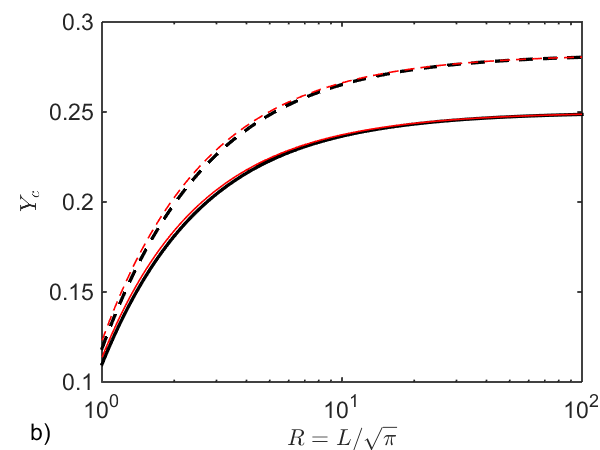}
\end{center}
\caption{Comparison of results of examples \ref{ex:cylinders} \& \ref{ex:SquareSquare} at different $R = L/\sqrt{\pi}$: a) $\lambda_c(L)$; b) $Y_c(L)$. Circular $\Omega_s$ is marked with the broken line and square $\Omega_s$ is marked with the solid line. Circular $\Omega$ marked in red and square $\Omega$ in black.}
\label{fig:squarepeg2}
\end{figure}

To summarise, these simple examples suggest that (for centrally placed convex) particles, when we have the same area of solid and the same area of fluid, the main differences in yield behaviour comes from the different perimeters of the particle. The optimal sets in $\Omega \setminus \Omega_s$ are selected such that $\lambda_c$ varies primarily with the area of $\Omega$ (and less significantly with its shape). For the same size of $\Omega$ (and $\Omega_s$) the particle with smaller perimeter has larger $Y_c$.
An illustration of the optimal sets for the square in square case is shown in Figure \ref{fig:SquareInSquare} (left) for $L=3.33$, for which we obtain $r = 0.600$ and $\abs{D u_r}(\Omega) = 5.67$.

\begin{example}[Influence of the aspect ratio and boundary] \label{ex:rectangle}

We revise example \ref{ex:SquareSquare}, keeping $\Omega$ as a square of side $L$ and replacing $\Omega_s$ by a centrally positioned rectangle of aspect ratio $\beta^2$, i.e.~the rectangle has height $\beta$ and width $\rev{1/}\beta \le L$. Provided that $\beta$ is sufficiently large there is a single Cheeger set in $\Omega\setminus \Omega_s$, given by
$\open{r}{\Omega}\setminus \Omega_s$ for some $r>0$. However, for sufficiently small $\beta$:
\[ \frac{1}{L}  \le \beta \le \frac{L}{2}\left( \sqrt{1 + \frac{8}{L^2} - 1} \right)   , \]
there may be a second Cheeger set configuration, as illustrated in Figure \ref{fig:rect_schematic}.

\begin{figure}
\begin{center}
\includegraphics[width = 8.5cm]{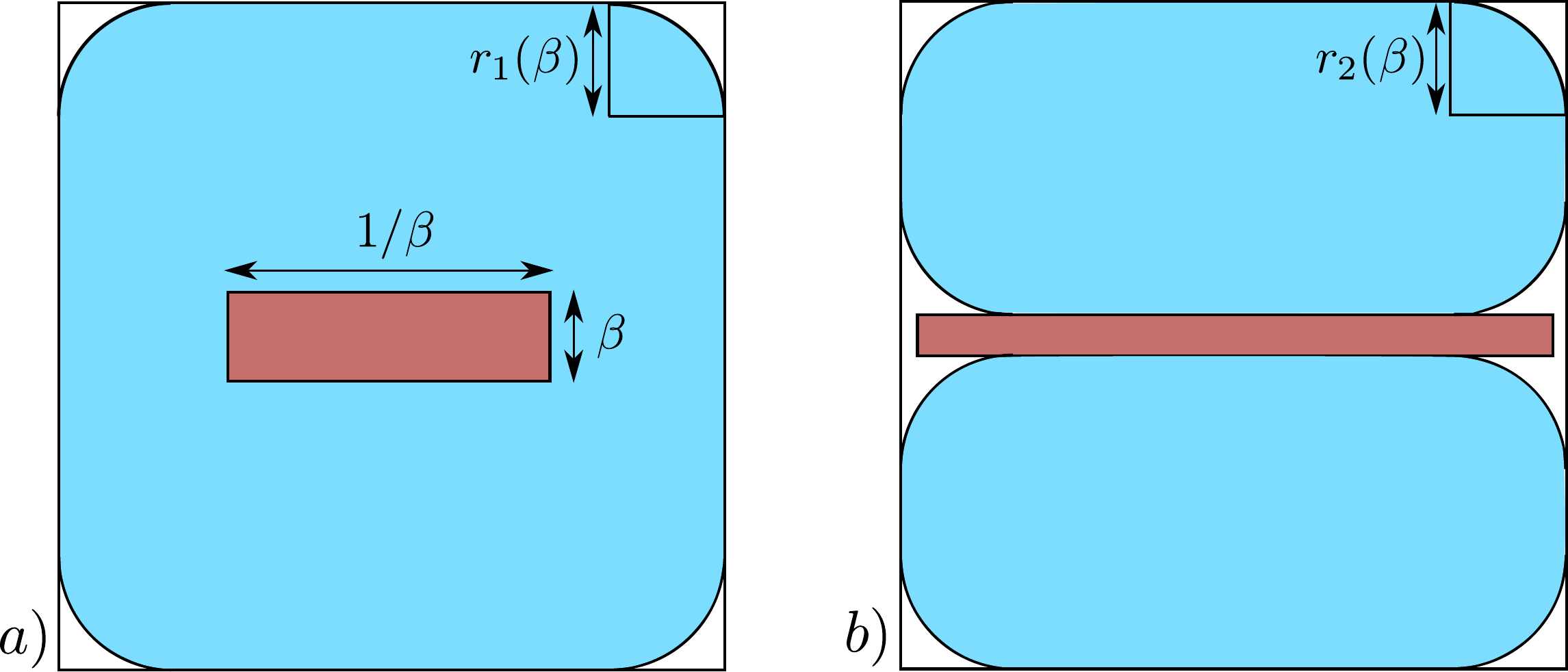}
\end{center}
\caption{Schematic of two different configurations for the rectangle with aspect ratio $\beta$: a) configuration 1; b) configuration 2.}
\label{fig:rect_schematic}
\end{figure}

For the first configuration we use Propositions \ref{prop:sameradius1} and \ref{prop:sameradius2} to find the radius $r_1(\beta) = 1/\lambda_{c,1}(\beta)$:
\[ r_1(\beta) = \frac{L}{2} \frac{1+\frac{\beta + 1/\beta}{2L}}{1-\pi/4}\left( 1 - \sqrt{1- (1-\pi/4)\frac{1-1/L^2}{(\frac{\beta + 1/\beta}{2L})^2} }\right)  . \]
The second configuration gives radius $r_2(\beta) = 1/\lambda_{c,2}(\beta)$:
\[ r_2(\beta) = \frac{3L-\beta}{8(1-\pi/4)} \left( 1 - \sqrt{1- 8(1-\pi/4)\frac{L(L-\beta)}{(3L-\beta)^2}} \right)  . \]

It is found that for a small band of $\beta$ the second configuration gives $\lambda_{c,2}(\beta)  < \lambda_{c,1}(\beta)$. In both cases we have $\Omega_{1c} = \Omega_s$ and the yield limit is
\[ Y_{c}(\beta) = \frac{1}{Per(\Omega_{1c}) + \min\{\lambda_{c,k}(\beta)\} |\Omega_{1c}|} =
\frac{1}{2(\beta + 1/\beta) + \min\{\lambda_{c,k}(\beta)\} }  .\]
The variation of $\lambda_{c}$ and $Y_c$ is illustrated in Figure \ref{fig:rectangle1} for $L=3$. Note that $Y_{c}(\beta)$ approaches the square in square results at $\beta = 1$. The difference between the two potential $Y_c$ in Figure \ref{fig:rectangle1}b is relatively small because for small $\beta$, $Per(\Omega_{s})$ becomes relatively large.

This example also serves to demonstrate geometric non-uniqueness. In the case that $\lambda_{c,2}(\beta)  < \lambda_{c,1}(\beta)$ either of the shaded regions above or below $\Omega_s$ in Figure \ref{fig:rect_schematic}b is a Cheeger set, as is the union. We may construct a minimizer of $\text{TV}$ over $\bvdo$ using the characteristic functions of either set, or any linear combination that satisfies the condition of zero flux. As commented earlier this non-uniqueness in $\bvdo$ stems from the geometric non-uniqueness. 

Interestingly, if one were to return to the original Bingham fluid problem and approach $Y \to Y_c^-$, the velocity solution is unique and can be shown to be symmetric, i.e.~the effect of viscosity here is to select a symmetric minimizer for $Y < Y_c$. 
\end{example}

\begin{figure}
\begin{center}
\includegraphics[width = 5cm]{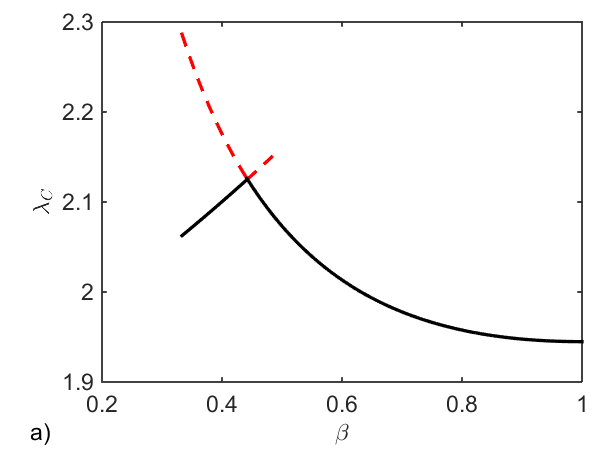}
\includegraphics[width = 5cm]{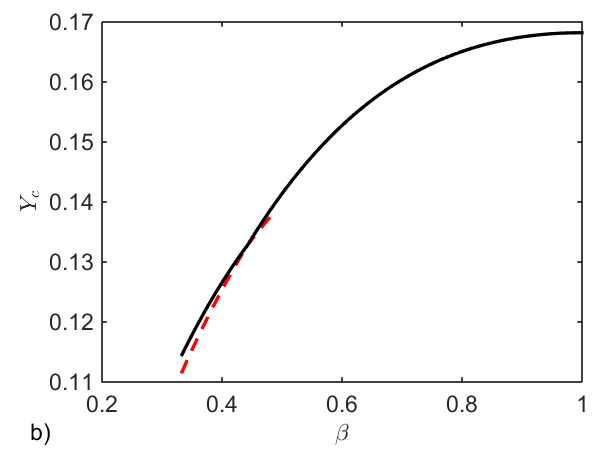}
\end{center}
\caption{Different mechanisms for the rectangle as $\beta$ is varied for $L=3$: a) $\lambda_c(\beta)$; b) $Y_c(\beta)$. The optimal values are in solid black and sub-optimal are in broken red.}
\label{fig:rectangle1}
\end{figure}

\begin{example}[Influence of the position of $\Omega_s$ with respect to the boundary] \label{ex:posbdy}
We revise example \ref{ex:SquareSquare} with $\Omega_s$ again being a square with
length $1$. This time  we move the inner square $\Omega_s$ in direction of $\partial \Omega$ and denote $d := d(\Omega_s,\partial \Omega)$.
The possible minimizers have
$\Omega_- = \open{r}{\Omega}\setminus \Omega_s$ or $\Omega_-=\open{r}{\Omega\setminus
\Omega_s}$ for some $r$, depending on $d$.
We illustrate this phenomenon in Figure \ref{fig:SquareInSquare}.
\end{example}

\begin{figure}
\begin{center}
\includegraphics[scale=0.85]{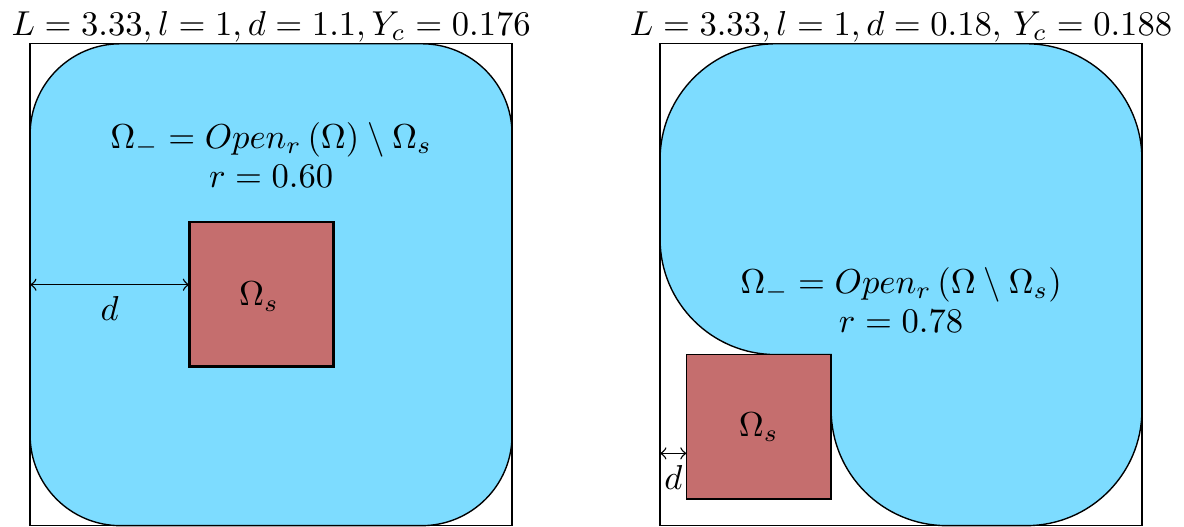}
\end{center}
\caption{In this case, area and perimeter of $\Omega,\Omega_s$ are constant.
We change the distance between $\partial \Omega$ and $\Omega_s$. The
critical yield number is larger if the inner set $\Omega_s$ is close to $\partial \Omega$.}
\label{fig:SquareInSquare}
\end{figure}

\subsection{Multiple particles}

We now consider multiple particles. In the first example, we retain the fixed $\abs{\Omega_s} = 1$ and consider the effects of increasing the number of particles. Intuitively, this increases the ratio of perimeter to area and hence we expect that $Y_c$ will reduce, as is indeed found to be the case. 

\begin{example}[A case with nontrivial $\Omega_1$]\label{ex:bridges}
We consider the two setups of Figure
\ref{fig:CompareDiamond2Squareb}, where for simplicity we keep $\Omega$ circular.
The flat regions correspond to the case where the optimal set  $\Omega_-$ is
equal to $\Omega \setminus \Omega_s$. 

We see that the orientation has an influence on the behavior of the minimizer as well as on the critical yield number. As $d$ is decreased below a critical value $\Omega_{1c}$ incorporates a \emph{bridge} between the two particles. The occurrence of the bridge clearly depends on orientation of the particles, and would also vary for different shaped particles. The phenomena of bridging between particles and of particles essentially acting independently beyond a critical distance have been studied computationally in the case of two spheres \cite{Liu2003,Merkak2006} (axisymmetric flows) and two cylinders \cite{Tokpavi2009} (planar two-dimensional flows). Aside from computed examples we know of no general theoretical results related to these phenomena, e.g. what the maximal distances for bridging are.

\begin{figure}
\begin{center}
\includegraphics[width = 0.325\textwidth]{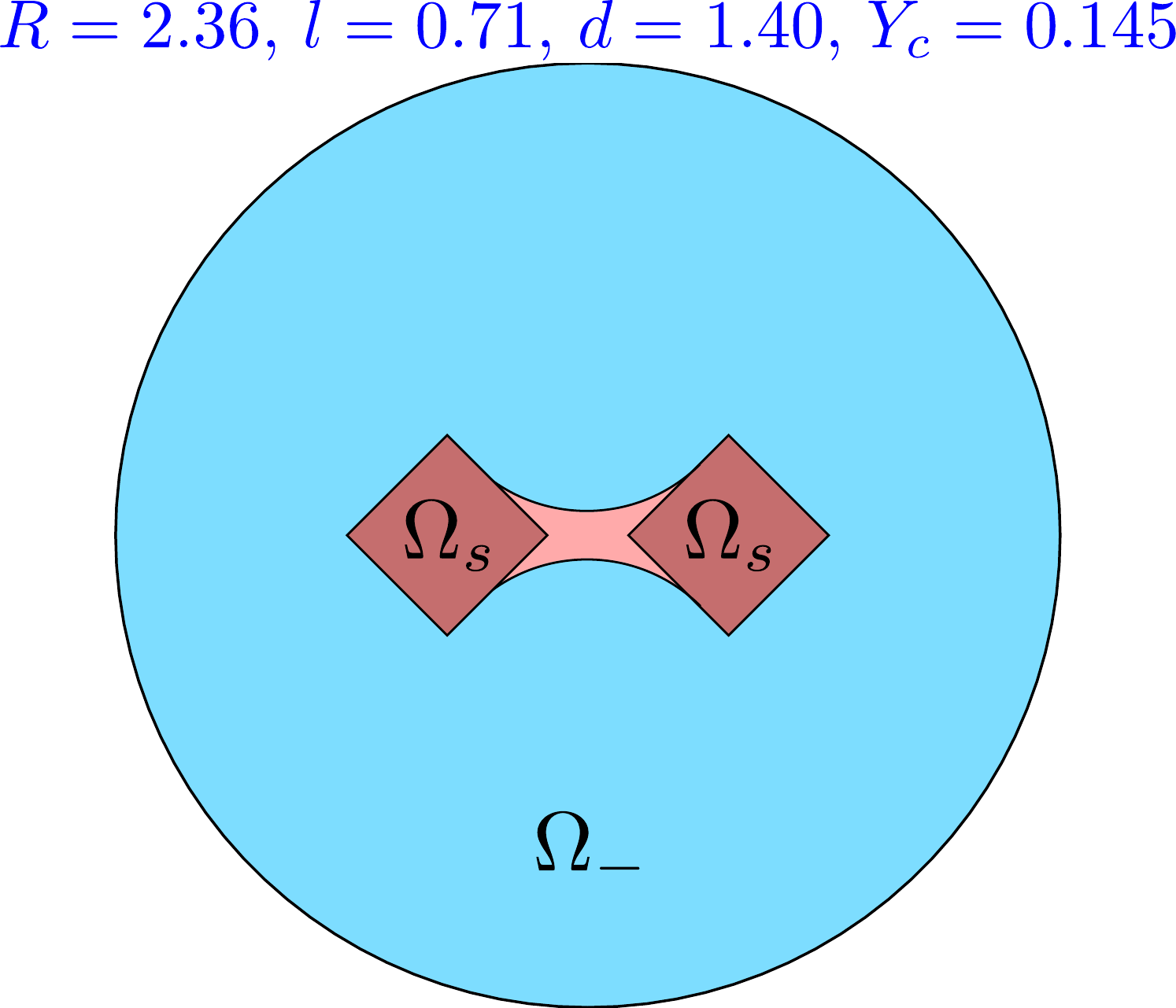}
\hfill
\includegraphics[width = 0.325\textwidth]{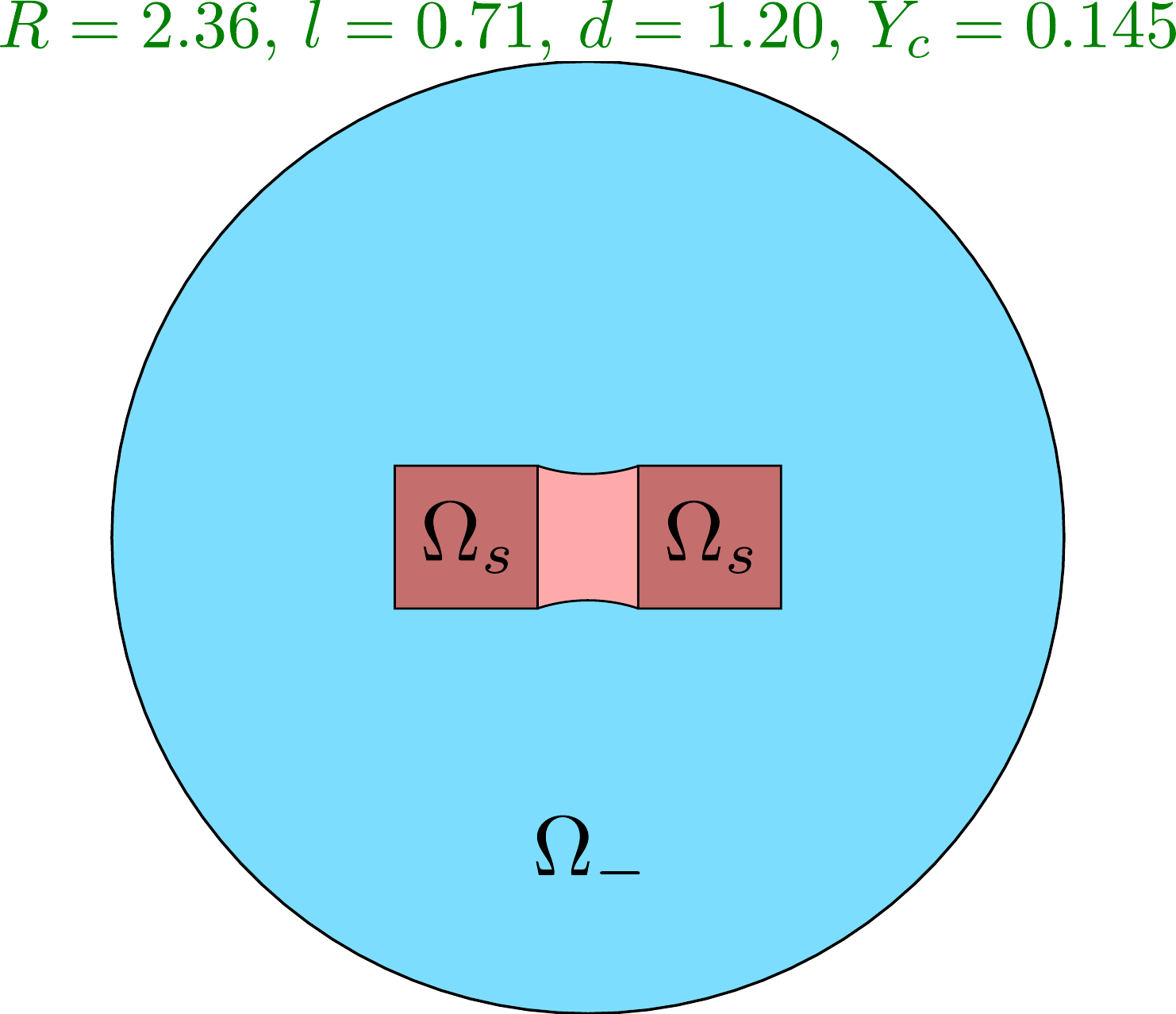}
\hfill
\includegraphics[width = 0.325\textwidth]{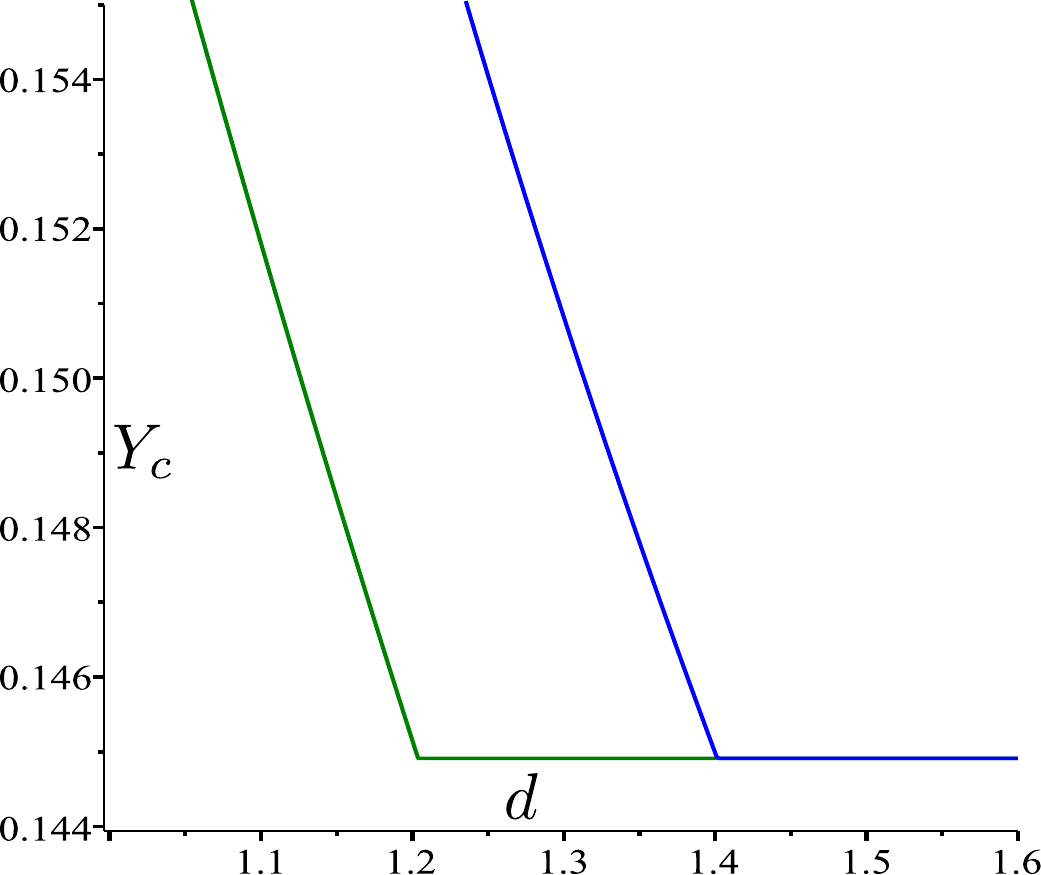}
\end{center}
\caption{Left and center: Two different arrangements of squares, at the corresponding transition points. Here, the trivial and nontrivial solutions coexist and the same critical yield number appears for both orientations of the square. Right: Critical yield numbers, with respect to the distance $d$ between the centers of the squares. The corners in the graph represent the transition between $\Omega_- = \open{r}{\Omega \setminus \Omega_s}$ and $\Omega_- = \Omega \setminus \Omega_s$.
}
\label{fig:CompareDiamond2Squareb}
\end{figure}

\end{example}

\begin{example}[Periodic arranged circles inside a square tube]\label{ex:tubes}
As a second example, we consider large arrays of particles, as illustrated in Figure \ref{fig:periodicCase}, i.e.~$\Omega$ is a square with length  $L$, and $\Omega_s$ is the union of $N^2$ small circles with radius $\delta$, the outermost of which are at distance $a$ from $\partial \Omega$. Here the intention is to illustrate particle size and separation effects and therefore we emphasize that in this case $|\Omega_s|$ is not constant for different $\delta$.

\begin{figure}
\begin{center}
\includegraphics[scale=0.8]{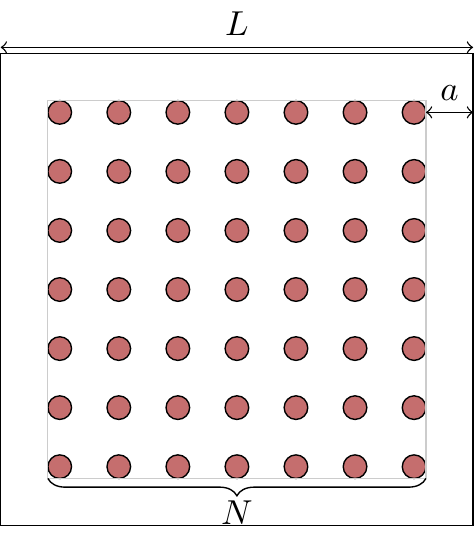}
\hspace{.05\textwidth}
\includegraphics[scale=0.25]{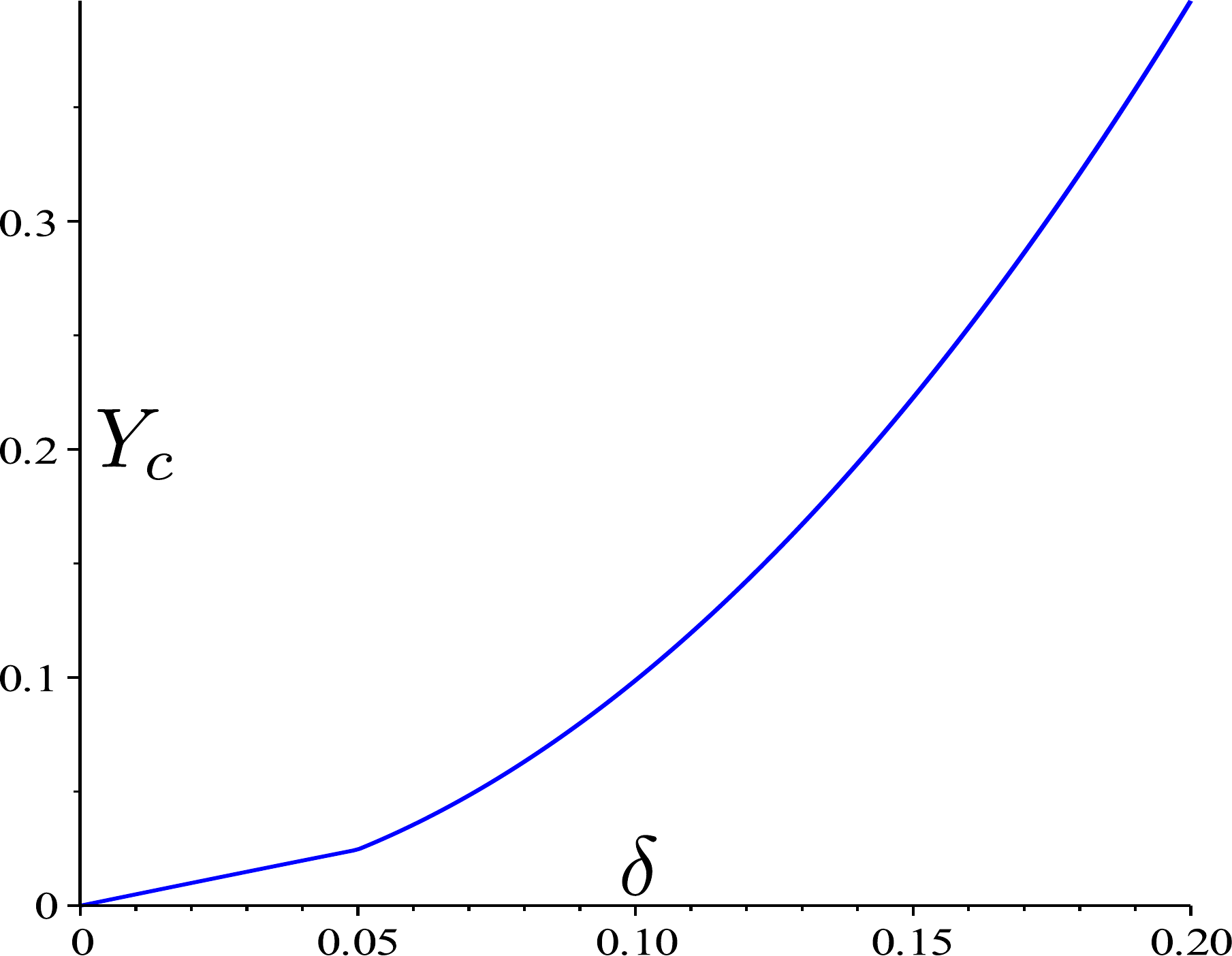}
\end{center}
\vspace{.01\textheight}
\begin{center}
\includegraphics[scale=0.85]{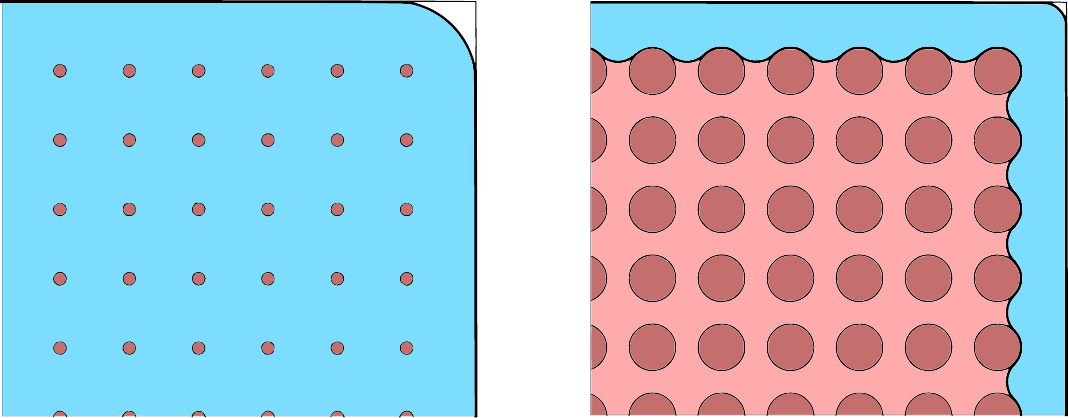}
\end{center}
 \caption{Upper row, left: Setup for the periodic case.
 Upper row, right: Dependence of the critical yield number on $\delta$, for $L=12$, $N=12$ and $a=0.4$. The corner in the graph corresponds to the transition from trivial to bridged optimal sets. Lower row: Optimal sets for $\delta=0.04$ and $\delta=0.2$, when $L=12$, $N=12$ and $a=0.4$.}
\label{fig:periodicCase}
\end{figure}

Two types of optimal sets appear: For $\delta$ small (left), we have $\Omega_1=\Omega_s$, $\Omega_-=\open{\lambda^{-1}}{\Omega} \setminus \Omega_s$. For bigger $\delta$ (right), one gets
$\Omega_1=\close{\lambda^{-1}}{\Omega_s}$, and $\Omega_-=\open{\lambda^{-1}}{\Omega\setminus \Omega_s}=\open{\lambda^{-1}}{\Omega}\setminus \Omega_1$ for $\lambda$ the corresponding Cheeger constant. One could think of a third configuration in which isolated components of $\Omega_-$ appear between the circles of $\Omega_s$, but it is easy to see that such a configuration has higher energy.
Figure \ref{fig:periodicCase} (top right) shows the variation in $Y_c$ with $\delta$ for a particular choice of parameters ($L=12$, $N=12$ and $a=0.4$). The observable \emph{kink} is where the transition between the two configurations occurs. 


Although this example is quite theoretical, this type of phenomenon occurs commonly in non-Newtonian suspension flows. In hydraulic fracturing, proppant suspensions are pumped along narrow fractures. For critical flow rates the individual dense proppant particles may act together in settling: so called \emph{convection}, see e.g.~\cite{Cleary1992}. This represents a serious risk for the process in that in \emph{convective} settling the group of particles settles faster than when individually settling, as in the latter case secondary flows are induced on a more local scale. It is interesting that these features (local and global) are captured by the simple model here, where the yield stress fluid definitively couples the particles via bridging. Convective settling is however not in general reliant on the yield stress.
\end{example}

These examples also expose an interesting question concerning individual particle behaviour. Dense suspensions in shear-thinning fluids often exhibit interesting settling patterns, e.g.~the column-like patterns in \cite{Daugan2004}. Such patterns are excluded in our study as we have assumed that the speed of $\Omega_s$ is uniform. There is a rich vein of interesting problems here to study. For example, if we remove the constraint of equal particle velocities, do particle arrays such as that considered above admit other optimal solutions that select patterns amongst the particles, e.g.~stripes moving at different speeds, or are slight perturbations from the regular lattice favourable?

\def\cprime{$'$}
\providecommand{\noopsort}[1]{}\def\ocirc#1{\ifmmode\setbox0=\hbox{$#1$}
\dimen0=\ht0
  \advance\dimen0 by1pt\rlap{\hbox to\wd0{\hss\raise\dimen0
  \hbox{\hskip.2em$\scriptscriptstyle\circ$}\hss}}#1\else {\accent"17 #1}\fi}

\section*{Acknowledgements}
This work has been supported by the Austrian Science Fund (FWF) within the national research network `Geometry+Simulation', project S11704.

\bibliographystyle{plain}
\bibliography{FriIglMerPoeSch16}
\end{document}